\documentclass[11pt]{article}
\usepackage[top=1in, bottom=1in, left=1in, right=1in]{geometry}
\usepackage{tikz}
\usetikzlibrary{shapes,arrows,matrix,patterns,positioning}
\usepackage{color}
\usepackage{graphicx}
\usepackage{algorithmic}
\usepackage{amssymb}
\usepackage{epstopdf}
\usepackage[norelsize,boxed,linesnumbered,vlined]{algorithm2e} 
\usepackage{float}
\usepackage{amsmath,amsthm,bm,color,epsfig,enumerate,caption}
\usepackage{hyperref}
\usepackage{booktabs}
\usepackage{multirow}
\usepackage{array}
\usepackage{todonotes}
\usepackage{xspace}

\newcommand{\hz}[1]{{\bf\color{black} #1 }}

\newtheorem{theorem}{Theorem}[section]

\newtheorem{lemma}[theorem]{Lemma}

\renewcommand{\appendix}[1]{
\section*{Appendix: #1}
}

\renewcommand{\O}{O}

\newcommand{\bbC}{\mathbb{C}}

\newcommand{\rID}{{\it rID}}
\newcommand{\cID}{{\it cID}}

\newcommand{\MHz}[0]{\,MHz\xspace}

\newcommand{\GB}[0]{\,GB\xspace}

\makeatletter
\newcommand*{\extendadd}{
  \mathbin{
    \mathpalette\extend@add{}
  }
}
\newcommand*{\extend@add}[2]{
  \ooalign{
    $\m@th#1\leftrightarrow$%
    \vphantom{$\m@th#1\updownarrow$}
    \cr
    \hfil$\m@th#1\updownarrow$\hfil
  }
}
\makeatother

\begin{document}

\title{A Hierarchical Butterfly LU Preconditioner for Two-Dimensional Electromagnetic Scattering Problems Involving Open Surfaces}
\author{Yang Liu \\ Lawrence Berkeley National Laboratory, USA\\ \href{mailto:liuyangzhuan@lbl.gov}{liuyangzhuan@lbl.gov}
     \and Haizhao Yang \\ Department of Mathematics\\ Purdue University\footnote{Current institute.}, USA\\National University of Singapore\footnote{Part of the work was done in Singapore.}, Singapore\\ \href{mailto:haizhao@purdue.edu}{haizhao@purdue.edu} }

\maketitle

\begin{abstract}
This paper introduces a hierarchical interpolative decomposition butterfly-LU factorization (H-IDBF-LU) preconditioner for solving two-dimensional electric-field integral equations (EFIEs) in electromagnetic scattering problems of perfect electrically conducting objects with open surfaces. H-IDBF-LU leverages the interpolative decomposition butterfly factorization (IDBF) to compress dense blocks of the discretized EFIE operator to expedite its application; this compressed operator also serves as an approximate LU factorization of the EFIE operator leading to an efficient preconditioner in iterative solvers. Both the memory requirement and computational cost of the H-IDBF-LU solver scale as $O(N\log^2 N)$ in one iteration; the total number of iterations required for a reasonably good accuracy scales as $O(1)$ to $O(\log^2N)$ in all of our numerical tests. The efficacy and accuracy of the proposed preconditioned iterative solver are demonstrated via its application to a broad range of scatterers involving up to $100$ million unknowns. 
\end{abstract}



\section{Introduction}
\label{sec:intro}
Iterative and direct surface integral equation (IE)
techniques are popular tools for the scattering analysis involving electrically large perfect electrically conducting (PEC)
objects. In iterative techniques, fast matvec algorithms including multilevel fast multipole algorithms
(MLFMA) \cite{ROKHLINFMM1,ROKHLINFMM2,CHENGFMM,FMM3,FMM4,FMM5,FMM6}, directional compression algorithms \cite{YingFFT,YingDirect,YingFMM,MESSNER20121175}, special function transforms \cite{AVERBUCH200019,BRADIE199394,768807,doi:10.1137/0913004,DEMANET2007368}, and butterfly factorizations (BFs) \cite{Butterfly1,Butterfly2,BF,IDBF,MBF,MVBF} can be used to rapidly apply discretized IE operators to trial solution
vectors. These techniques typically require $O ( {N_{i}^{\alpha}}N \log^\beta N )$ (${\alpha,}~\beta = 1$ or $2$) CPU and memory resources, where $N$ is the
dimension of the discretized IE operator and ${N_i}$ is the number of iterations required for
convergence. The widespread adoption and success of fast iterative methods for solving
real-world electromagnetic scattering problems can be attributed wholesale to their low
computational costs when ${N_i}$ is small. However, in many applications when scatterers support resonances or are discretized using multi-scale/dense meshes, the corresponding linear system becomes ill-conditioned {especially for first-kind integral equation formulations} and ${N_i}$ can be prohibitively large. This motivates a significant amount of research devoted to the design of efficient semi-analytic or analytic preconditioners for integral equations for both open and closed surfaces \cite{pEFIE2D1,pEFIE2D2,BrunoCalderon1,BrunoCalderon1,Lexing2015}. 
 
Direct solvers do not suffer (to the same degree) from the aforementioned
drawbacks, since they construct a compressed representation of the inverse of the discretized
IE operator directly. Most state-of-the-art direct methods apply low rank (LR) approximations to compress judiciously selected off-diagonal blocks
of the discretized IE operator and its inverse \cite{doi:10.1029/2012RS004988,RS2,5659467,4589136,5175469,Bebendorf2005,doi:10.1137/16M1095949}. LR
compression schemes are highly efficient and lead to nearly linear scaling direct solvers for electrically
small \cite{RS2,CORONA2015284}, or specially-shaped structures \cite{MARTINSSON2007288,533187,5297271,6824171}.
However, for electrically large and arbitrarily shaped scatterers, the numerical rank of blocks of the discretized IE operators and its inverse is no longer small. As a result, there is no theoretical guarantee of low computational costs for LR schemes in this high-frequency regime;
experimentally their CPU and memory requirements have been found to scale as
$O ( N^\alpha \log^\beta N )$ ($\alpha \in[ 2.0,3.0]$, $\beta\geq 1$) and $O(N^\alpha \log N )$ ($\alpha \in[1.3,2.0]$), respectively. More recently, butterfly factorizations have been applied to construct reduced-complexity direct solvers in the high-frequency regime \cite{HSSBF,LUBF,LUBF2}. These solvers construct butterfly factorizations with constant ranks for blocks (that are LR less compressible) in the discretized IE operator and its inverse and rely on fast randomized algorithms to construct the inverse in $O(N^{1.5}\log N)$ operations.   

The lack of quasi-linear complexity direct solvers in the high-frequency regime motivates this work to develop a hierarchical butterfly-compressed algebraic preconditioner for solving EFIEs with $O ( N \log^2 N )$ CPU and memory complexity per iteration and up to $O(\log^2N)$ total iterations, for analyzing scattering from electrically large
two-dimensional PEC objects with open surfaces. First, the interpolative decomposition butterfly factorization (IDBF) algorithm \cite{IDBF} is used to compress off-diagonal blocks of the discretized IE operator, leading to $O(N\log^2N)$ construction and application algorithms with the hierarchical IDBF (H-IDBF) of the IE operator. Second, we construct an approximate butterfly-compressed LU factorization of H-IDBF inspired by the work in \cite{LUBF}. In contrast to \cite{LUBF} that computes the LU factorization via randomized butterfly algebra, here the lower and upper triangular parts of the H-IDBF can directly serve as its approximate LU factorization. This is justified by the observation that the discretized IE operator and its LU factors exhibit similar oscillation patterns after a proper row-wise/column-wise ordering. This approximate LU factorization permits construction of a quasi-linear complexity preconditioner for EFIEs when applied to a wide range of scatterers. Compared to the analytic preconditioners in \cite{BrunoCalderon1,Lexing2015}, the proposed H-IDBF-LU preconditioner is capable of analyzing scattering from electrically large objects involving up to $100$ million unknowns.

\section{Interpolative Decomposition Butterfly Factorization (IDBF)}
\label{sec:IDBF}

Since the IDBF will be applied repeatedly in this paper, we briefly sumarize the  $O(N\log N)$ IDBF algorithm proposed in \cite{IDBF} for a complementary low-rank matrix $K\in\mathbb{C}^{M\times N}$ with $M\approx N$. 

Let $X$ and $\Omega$ be the row and column index sets of $K$. Two trees $T_X$ and $T_\Omega$ of the same depth $L=O(\log N)$,
associated with $X$ and $\Omega$ respectively,
are constructed by dyadic partitioning with approximately equal node sizes with leaf node sizes no larger than $n_0$. Denote the root level of the tree as level $0$ and
the leaf one as level $L$. Such a matrix $K$ of size $M\times N$
is said to satisfy the {\bf complementary low-rank property} if for
any level $\ell$, any node $A$ in $T_X$ at level $\ell$, and any node
$B$ in $T_\Omega$ at level $L-\ell$, the submatrix $K_{A,B}$, obtained
by restricting $K$ to the rows indexed by the points in $A$ and the
columns indexed by the points in $B$, is numerically low-rank.
See Figure \ref{fig:submatrices} for an illustration of low-rank submatrices
in a complementary low-rank matrix of size $16n_0\times 16n_0$.

\begin{figure}[ht!]
  \begin{center}
    \begin{tabular}{ccccc}
      \includegraphics[height=1.1in]{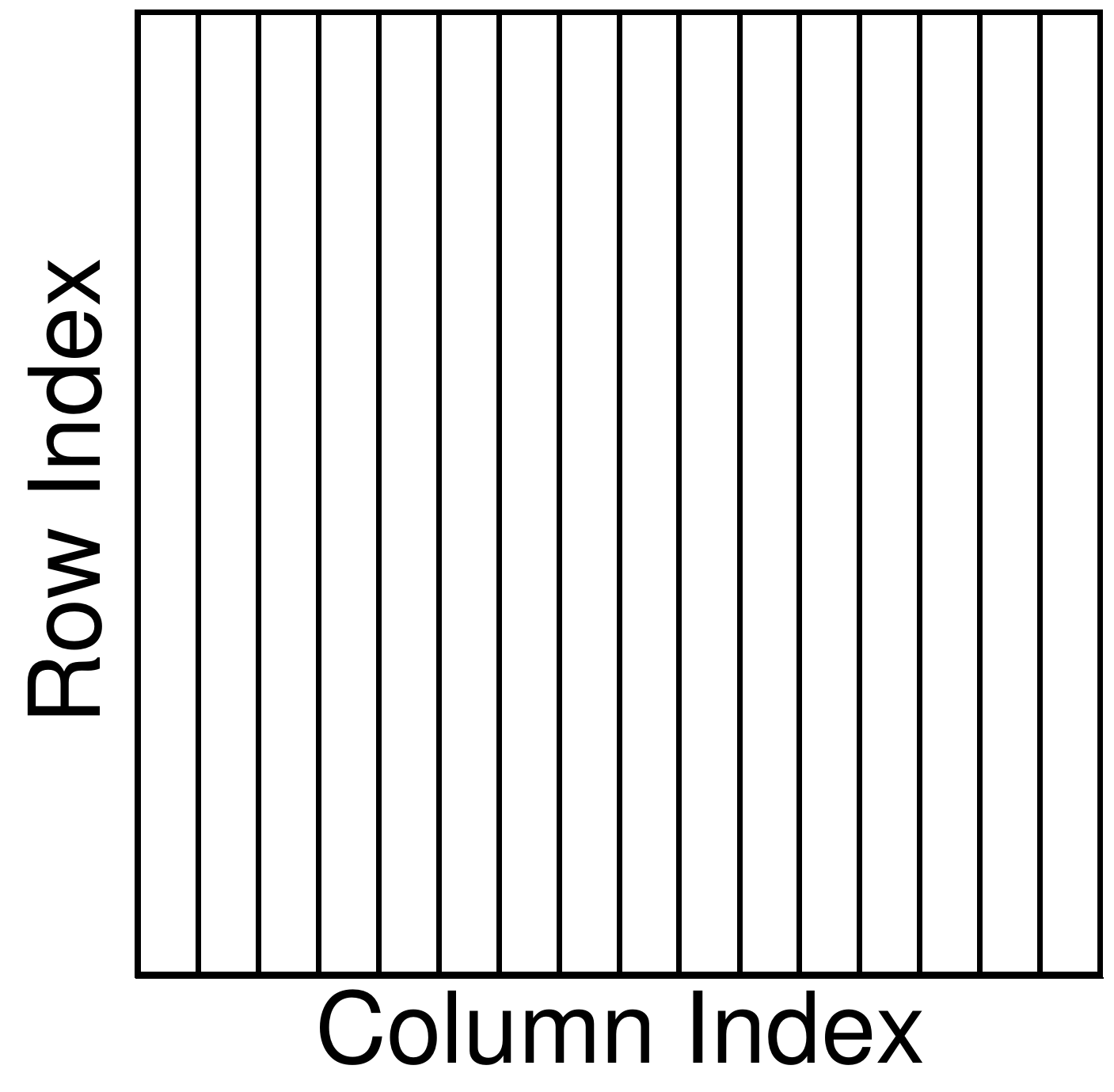}&
      \includegraphics[height=1.1in]{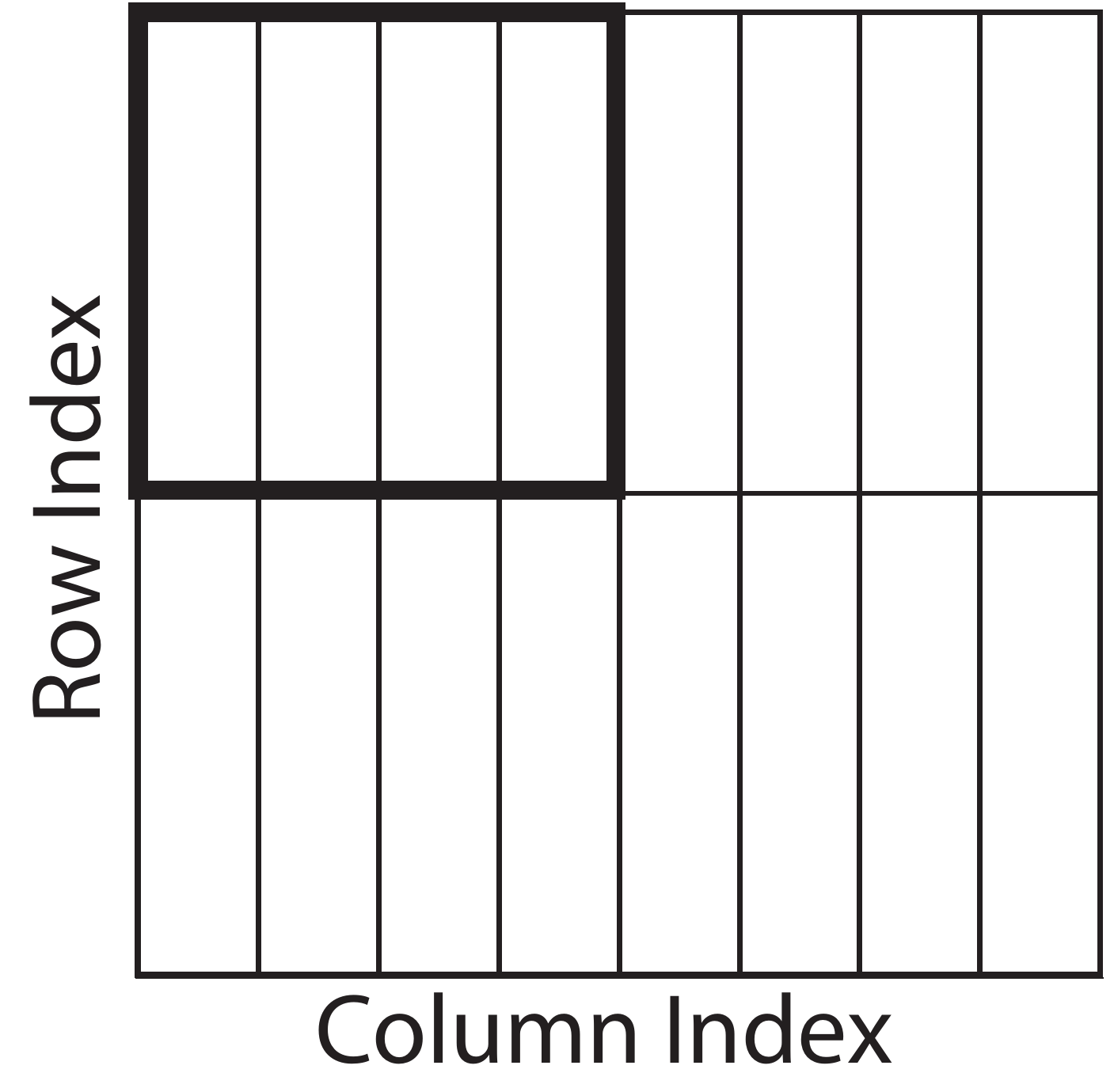}&
      \includegraphics[height=1.1in]{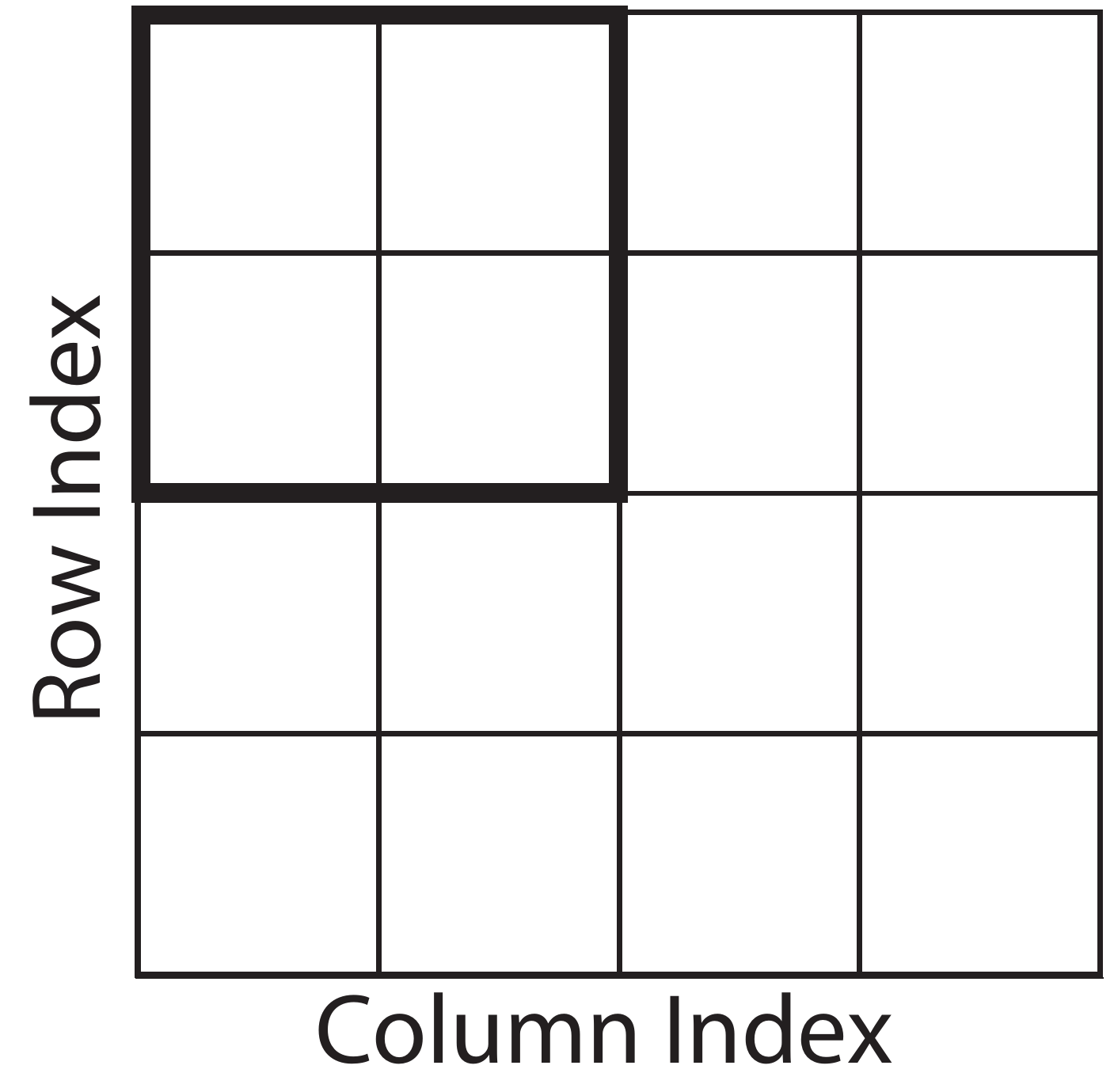}&
      \includegraphics[height=1.1in]{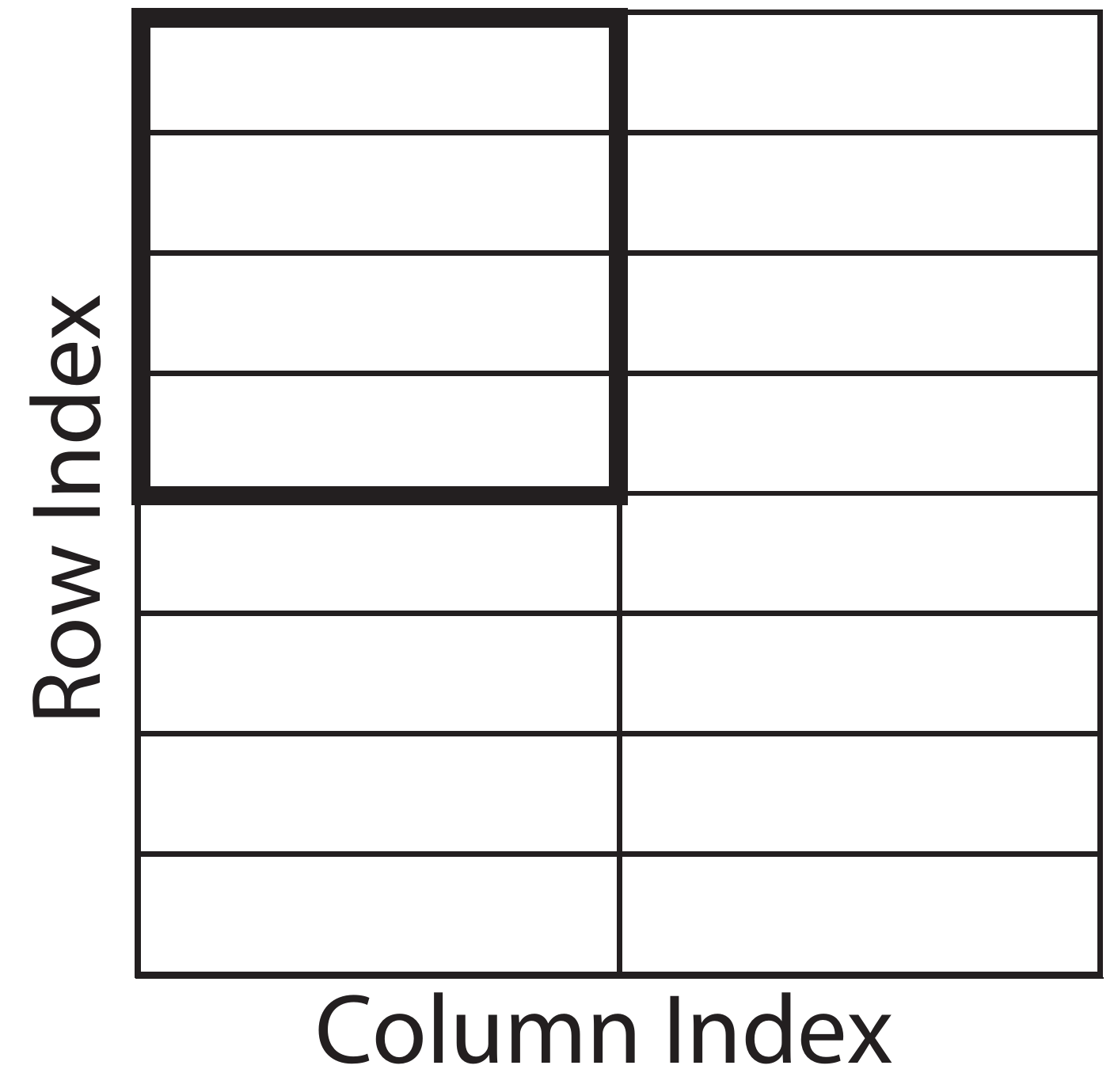}&
      \includegraphics[height=1.1in]{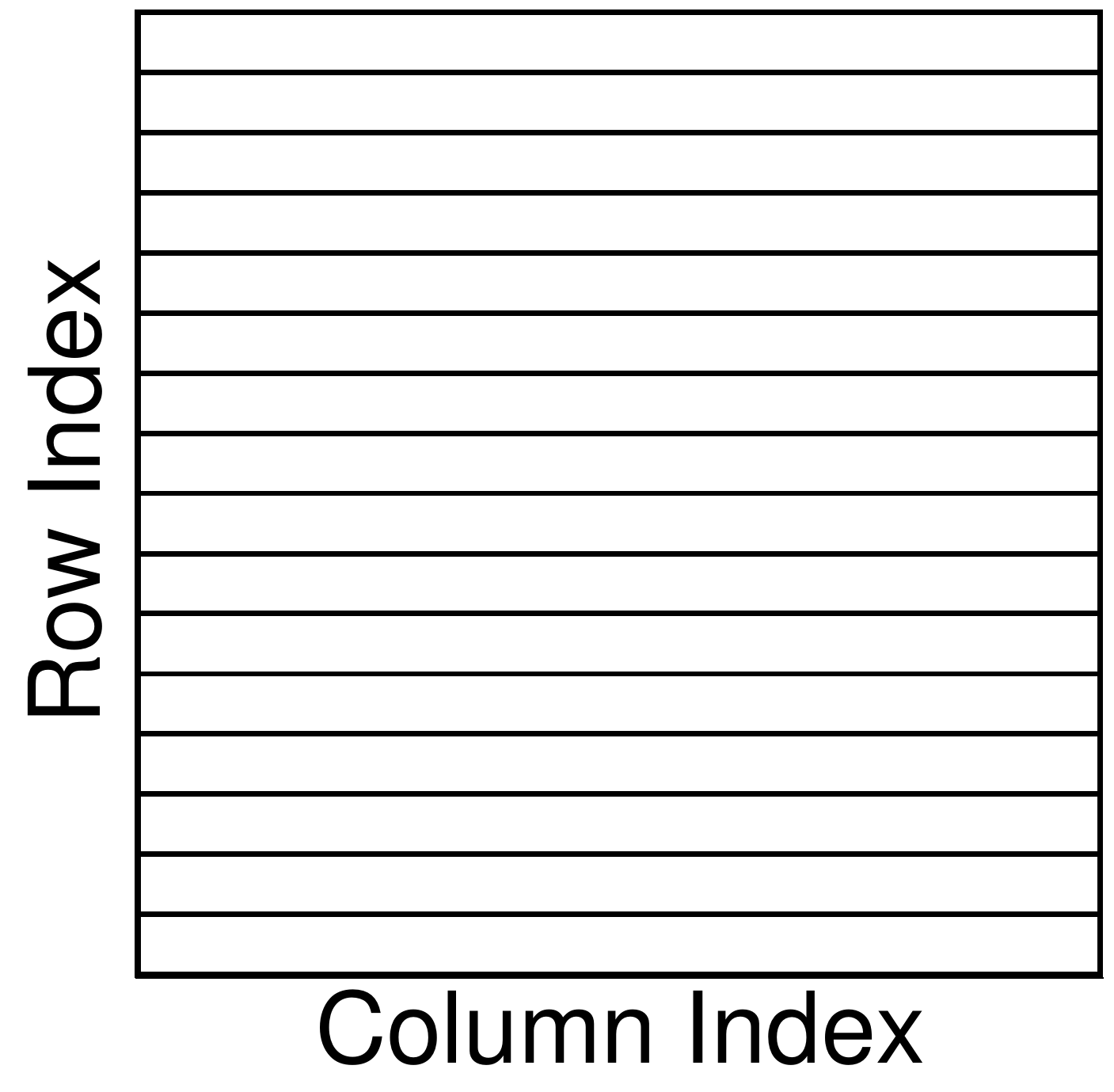}
    \end{tabular}
  \end{center}
  \caption{Hierarchical decomposition of the row and column indices
    of a complementary low-rank matrix of size
     $16n_0\times 16n_0$.  The trees $T_{X}$ ($T_{\Omega}$)
    has a root containing $16n_0$ column (row) indices and leaves
    containing $n_0$ row (column) indices.  The rectangles above
    indicate the low-rank submatrices that will be factorized in IDBF.}
\label{fig:submatrices}
\end{figure}

Given $K$, or equivalently an $O(1)$ algorithm to evaluate an arbitrary entry of $K$, IDBF aims at constructing a data-sparse representation of $K$ using the ID of low-rank submatrices in the complementary low-rank structure (see Figure \ref{fig:submatrices}) in the following form: 
\begin{equation}
\label{eqn:BFM}
K \approx U^{L}U^{L-1}\cdots U^{h} S^h V^{h}\cdots V^{L-1}V^{L},
\end{equation}
where the depth $L=\O(\log N)$ is assumed to be even without loss of generality,
$h=L/2$ is a middle level index, and all factors are sparse matrices with
$\O(N)$ nonzero entries. Storing and applying IDBF requires only $O(N\log N)$ memory and time.

{The construction of the butterfly factorization was usually expensive though the application is cheap. The IDBF in \cite{IDBF} applies a linear scaling interpolation decomposition technique with carefully selected matrix skeleton of $K$ to achieve $O(N\log(N))$ factorization time, making the butterfly factorization more attractive in large-scale applications. Hence, the H-IDBF-LU to be introduced in the next section also admits nearly linear scaling factorization time.}

\section{H-IDBF-LU}

\subsection{Overview of the H-IDBF-LU preconditioner}
\label{sub:mov}

\begin{figure}[ht!]
  \begin{center}
    \begin{tabular}{c}
      \includegraphics[height=2in]{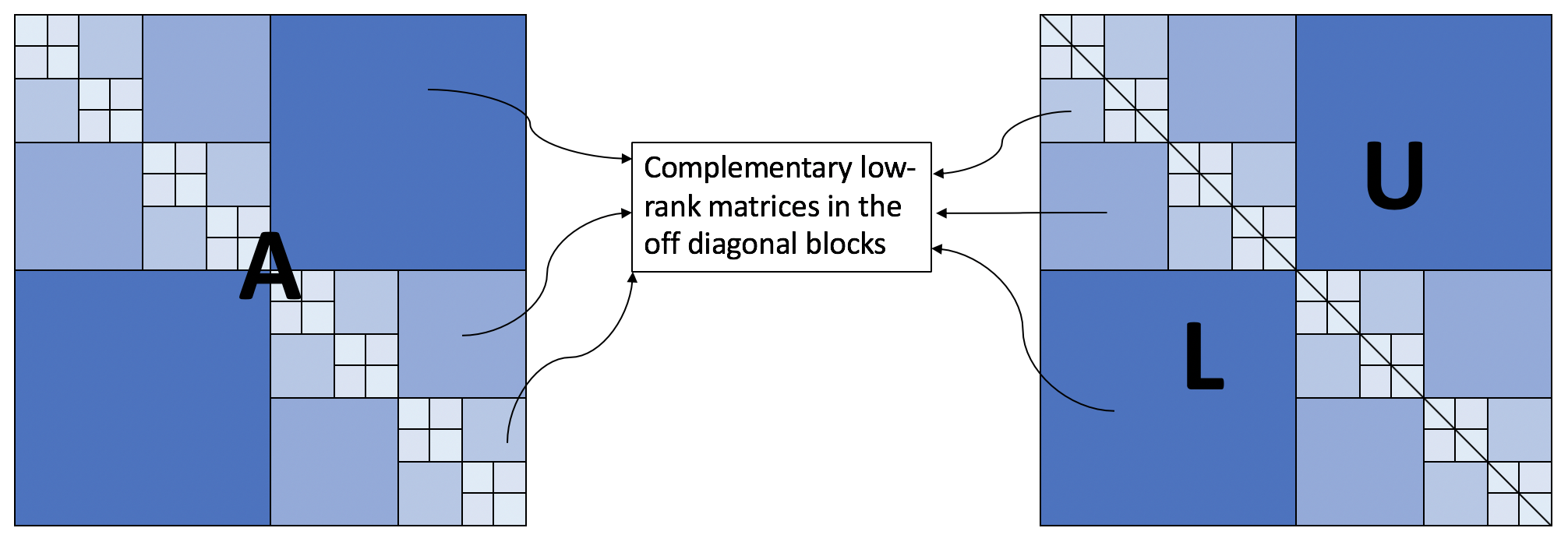}
    \end{tabular}
  \end{center}
\caption{Left: hierarchical decomposition of the impedance matrix $A$ with off-diagonal blocks as complementary low-rank matrices. Right: $A$ is separated into a lower triangular part (denoted as $L$) with diagonal entries being $1$, and a upper triangular part (denoted as $U$) with diagonal entries equal to those of $A$. }
\label{fig:H-IDBF-LU}
\end{figure}

Here we briefly describe the motivation and workflow of H-IDBF-LU in the context of EFIE for 2D $\mathrm{TM_z}$ scattering. Let $S$ denote a PEC cylindrical shell residing in free space. A current $J_z(\rho)$ is induced on $S$ by an incident electric field $E^{inc}_z(\rho)$. Enforcing a vanishing total electric field on $S$ results in the following EFIE:
\[
E^{inc}_z(\rho) = \frac{{\kappa}\eta_0}{4}\int_S J_z(\rho')H^{(2)}_0({\kappa}|\rho-\rho'|)ds, \quad \forall\rho\in S.
\]
Here, ${\kappa}=2\pi/\lambda_0$ is the wavenumber, $\lambda_0$ represents the free-space wavelength, $\eta_0$ is the intrinsic impedance of free space, and $H^{(2)}_0$ denotes the zeroth-order Hankel function of the second kind. Upon discretization of the current with pulse basis functions and point matching the above EFIE, the following system of linear equations is attained \cite{Butterfly1}
\begin{equation}\label{eqn:Axb}
Ax=b,
\end{equation}
where $A$ is the impedance matrix with
\begin{align*}
 A_{ij} &=
  \begin{cases}
   {\frac{{\kappa}\eta_0w_j}{4}} H_0^{(2)}({\kappa}|\rho_i-\rho_j|),        & \text{if } i\neq j, \\
   {\frac{{\kappa}\eta_0w_i}{4}}\left[ 1-\mathrm{i} \frac{2}{\pi}\ln \left( \frac{\gamma {\kappa}{w_i}}{4e} \right)  \right],       & \text{otherwise},
  \end{cases}
\end{align*}
$b_i=E^{inc}_z(\rho_i)$ and $x_i=J_z(\rho_i)$ where $e\approx 2.718$, $\gamma\approx 1.781$, {$w_i$} and $\rho_i$ are the length and center of the line segment associated with $i$-th pulse basis function. In the proposed solver, the linear system in \eqref{eqn:Axb} is rescaled to $aAx=ab$, where we choose $a:=\frac{1}{\max_i\{|A(i,1)|\}}$ or $a:=\left\lVert A\right\rVert_2$. Here $\left\lVert A\right\rVert_2$ can be rapidly computed via randomized SVD algorithms. Rescaling is a key step for the proposed LU preconditioner to avoid arithmetic overflow in solution vectors as the iteration count increases. {Note that for second-kind integral equations the rescaling may not be necessary.} With the abuse of notations, we will still use $Ax=b$ to denote the linear system after rescaling.

It was validated in \cite{Butterfly1,HSSBF} that {$A$} admits a hierarchical complementary low-rank property, i.e., off-diagonal blocks (only well-separated ones if necessary) are complementary low-rank matrices (See Figure \ref{fig:H-IDBF-LU} (left) for an example). \cite{Butterfly1} originally proposed a hierarchical compression technique with $O(N\log^2 N)$ operations to compress the impedance matrix $A$; the IDBF in \cite{IDBF} proposed a more stable algorithm with the same accuracy to compress $A$. After compression, it requires $O(N\log^2N)$ operations to apply the impedance matrix and makes it possible to design efficient linear solvers for \eqref{eqn:Axb}. The algorithm in \cite{IDBF} for compression of the impedance matrix is referred to as hierarchical IDBF (H-IDBF) in this paper.


Since the H-IDBF of the impedance matrix $A$ leads to an $O(N\log^2N)$ fast matvec algorithm in a purely algebraic fashion, an immediate question is how to construct a direct solver or an efficient preconditioner for the linear system in \eqref{eqn:Axb}. It was observed (without proof) in \cite{LUBF} that the $L$ and $U$ factors in the LU factorization of $A$ are also hierarchically complementary low-rank matrices. 
Hence, \cite{LUBF} proposed an $O(N^{1.5}\log N)$ factorization algorithm based on randomized butterfly algebra to construct hierarchical compressions of the $L$ and $U$ factors, followed by an $O(N\log^2N)$ block triangular solution algorithm for rapidly application of the matrix inverse to a vector. The lack of a quasi-linear LU factorization algorithm as a direct solver in \cite{LUBF} motivates this work to construct an approximate LU factorization of $A$ in $O(N\log^2N)$ operations as an efficient preconditioner. 

The main observation of the proposed approximate LU factorization of $A$ is that, the lower and upper triangular parts of $A$ have similarly oscillation patterns as the $L$ and $U$ factors of its LU factorization, when the scatterer consists of open surfaces that do not support high-Q resonance. It is also critical that the rows and columns of $A$ are properly ordered such that $A$ exhibits only $O(1)$ discontinuities {with respective to $N$} in each row and column. Figure \ref{fig:ex1} and \ref{fig:ex2} show two examples of such scatterers; one is a spiral scatterer and the other one consists of two parallel strips. To be precise, define triangular matrices $\tilde{L}$ and $\tilde{U}$ as follows:
\begin{align}\label{eqn:Lt}
 \tilde{L}(i,j) &=
  \begin{cases}
   0,        & \text{if } i<j, \\
   1,     & \text{if } i=j,\\
   A(i,j),       & \text{otherwise},
  \end{cases}
\end{align}
and
\begin{align}\label{eqn:Ut}
 \tilde{U}(i,j) &=
  \begin{cases}
   0,        & \text{if } i>j, \\
   A(i,j),       & \text{otherwise}.
  \end{cases}
\end{align}
Let $L$ and $U$ be the lower and upper triangular matrices of the LU factorization of $A$, we observe $L\approx \tilde{L}$, $U\approx \tilde{U}$, and $\tilde{L}^{-1}A\tilde{U}^{-1}\approx I$ numerically, where $I$ is an identity matrix. See Figure \ref{fig:ex1} and \ref{fig:ex2} for the visualization of the approximation $\tilde{L}^{-1}A\tilde{U}^{-1}\approx I$ and the eigenvalue distribution of $\tilde{L}^{-1}A\tilde{U}^{-1}\approx I$. Take a spiral scatterer for example, the real (and imaginary) parts of $A$ (Figure \ref{fig:ex1}(d)) and its LU factors (Figure \ref{fig:ex1}(e)) exhibit very similar oscillation patterns. Before preconditioning, there are eigenvalues of $A$ clustered at the origin (Figure \ref{fig:ex1}(b)), while the preconditioned system $\tilde{L}^{-1}A\tilde{U}^{-1}$ has no eigenvalues near the origin (Figure \ref{fig:ex1}(c)). In fact, $\tilde{L}^{-1}A\tilde{U}^{-1}$ exhibits vanishingly small off-diagonal values (Figure \ref{fig:ex1}(f)). The H-IDBFs of $\tilde{L}$ and $\tilde{U}$ are referred to as the H-IDBF-LU of the impedance matrix $A$ in this paper.

Note that $\tilde{L}$ and $\tilde{U}$ are the lower and upper triangular parts of $A$, whose H-IDBFs are readily available from the H-IDBF of $A$ without any extra cost. Following the block triangular solve algorithm for H-matrices \cite{lecture}, it only requires $O(N\log^2N)$ operations to apply $\tilde{L}^{-1}$ and $\tilde{U}^{-1}$ to a vector. Hence, the approximate LU factorization results in an $O(N\log^2N)$ preconditioner for the linear system in \eqref{eqn:Axb}. To be more specific, instead of directly solving \eqref{eqn:Axb} using traditional iterative solvers, we solve a preconditioned linear system
\begin{equation}\label{eqn:Axb2}
\tilde{L}^{-1}A\tilde{U}^{-1}y=\tilde{L}^{-1}b,
\end{equation}
and
\begin{equation}\label{eqn:Axb3}
\tilde{U}x=y.
\end{equation}
Since the eigenvalues of $\tilde{L}^{-1}A\tilde{U}^{-1}$ are well separated from $0$ and gathered around $1$ (see Figure \ref{fig:ex1} and \ref{fig:ex2} for two examples), the condition number of $\tilde{L}^{-1}A\tilde{U}^{-1}$ is well controlled and the number of iterations for solving \eqref{eqn:Axb2} is small. In all of our examples, the iteration number is $O(1)$, $O(\log N)$, or $O(\log^2N)$ for scatterers consisting of open surfaces. As the construction, application and triangular solve algorithms all require $O(N\log^2N)$ operations, the overall preconditioner is also quasi-linear.

\subsection{\hz{Algorithm description}}

The remaining task in this section is to introduce the fast construction, application and solution phases of the H-IDBF-LU solver. 

\subsubsection{Construction and application of H-IDBF}
\label{sub:conh}

As we have seen in the previous subsection, the construction of the proposed preconditioner, i.e., the H-IDBFs of $\tilde{L}$ and $\tilde{U}$, is an immediate result of the H-IDBF of $A$. Hence, it is sufficient to introduce the construction and application of the H-IDBF of $A$. In fact, these resemble standard techniques for hierarchical matrices \cite{Grasedyck2003,Hmatrix,HSSmatrix}, and can be summarized in Algorithm \ref{alg:1} and \ref{alg:2}, respectively. For illustration purpose, we assume $N$ is a power of $2$ and omit the parameters of IDBF in all algorithms. Furthermore, we assume all off-diagonal blocks of $A$ are butterfly compressible (i.e., weak admissibility). However, the proposed solver trivially extends to arbitrary problem sizes and strong admissibility. Let $n_0$ denote the predefined leaf-level size parameter for all IDBFs, algorithm \ref{alg:1} constructs a IDBF of $L=\log_2 N - \log_2 n_0 -v$ levels for each off-diagonal block of dimension $N/2^v$. Once constructed, algorithm \ref{alg:2} can apply the compressed $A$ to arbitrary vectors by accumulating the matvec results of each submatrix.  

\vspace{0.25cm}
\begin{algorithm}[H]
 \KwData{A hierarchical complementary low-rank matrix $A\in\mathbb{C}^{N\times N}$ or an algorithm to evaluate an arbitrary entry of $A$ in $O(1)$ operations.}
 \KwResult{The H-IDBF of $A$, denoted as $F$, stored in a data structure consisting of four parts: $F_{11}$, $F_{12}$, $F_{21}$, and $F_{22}$. }
  \eIf{$N\leq n_0$ }{
  \For{$i=1,2$ and $j=1,2$}{
   Let $F_{ij}=A((1:N/2)+(i-1)N/2,(1:N/2)+(j-1)N/2)$\;
  }
   }{
   Recursively apply Algorithm \ref{alg:1} with $A(1:N/2,1:N/2)$ as the input to obtain $F_{11}$ as the corresponding output\;
   Recursively apply Algorithm \ref{alg:1} with $A(N/2+1:N,N/2+1:N)$ as the input to obtain $F_{22}$ as the corresponding output\;
   Construct the IDBF of $A(1:N/2,N/2+1:N)$ and store it in $F_{12}$\;
   Construct the IDBF of $A(N/2+1:N,1:N/2)$ and store it in $F_{21}$\;
  }
 \caption{An $O(N\log^2N)$ recursive algorithm for constructing the H-IDBF of a hierarchical complementary low-rank matrix $A$. }
 \label{alg:1}
\end{algorithm}

\vspace{0.25cm}
\begin{algorithm}[H]
 \KwData{A H-IDBF of $A\in\mathbb{C}^{N\times N}$ and vectors $v,u\in\mathbb{C}^N$.}
 \KwResult{$u\mathrel{+}=Av$.}
 \eIf{$N\leq n_0$ }{
  \For{$i=1,2$}{
   $u((1:N/2)+(i-1)N/2)\mathrel{+}=F_{i1}u(1:N/2)+F_{i2}u(N/2+1:N)$\;
  }
   }{
   Recursively apply Algorithm \ref{alg:2} to compute $u(1:N/2)\mathrel{+}=F_{11}v(1:N/2)$\;
   Recursively apply Algorithm \ref{alg:2} to compute $u(N/2+1:N)\mathrel{+}=F_{22}v(N/2+1:N)$\;
   Apply the IDBF stored as $F_{21}$ to get $u(N/2+1:N)\mathrel{+}=F_{21}v(1:N/2)$\;
   Apply the IDBF stored as $F_{12}$ to get $u(1:N/2)\mathrel{+}=F_{12}v(N/2+1:N)$\;
  }
 \caption{An $O(N\log^2N)$ recursive algorithm for applying the H-IDBF of a hierarchical complementary low-rank matrix $A$ to a vector $v$.}
 \label{alg:2}
\end{algorithm}

\subsubsection{Solution of H-IDBF-LU}
\label{sub:prec}

\begin{figure}[ht!]
  \begin{center}
    \begin{tabular}{ccc}
      \includegraphics[height=1.6in]{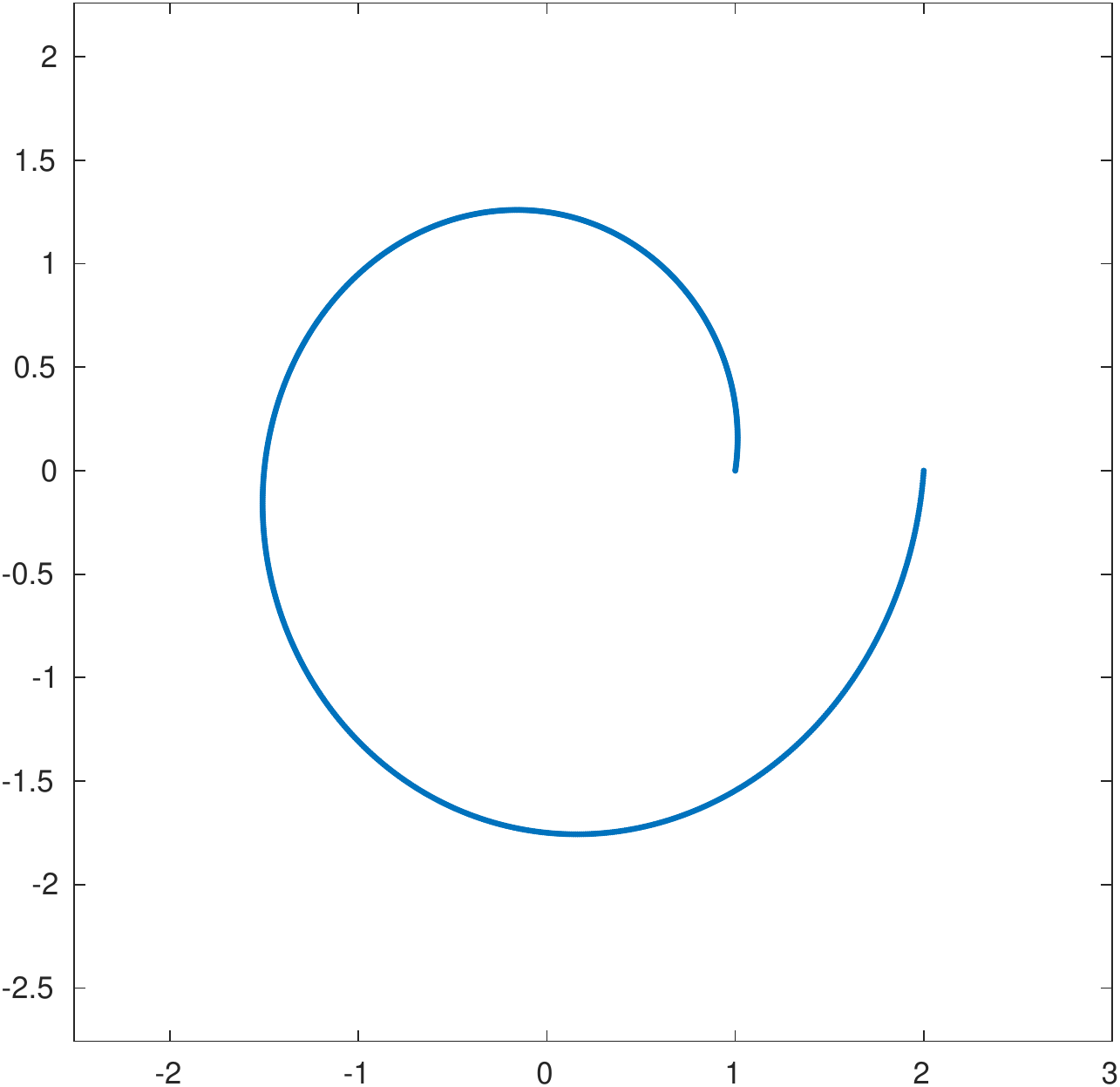} &  \includegraphics[height=1.6in]{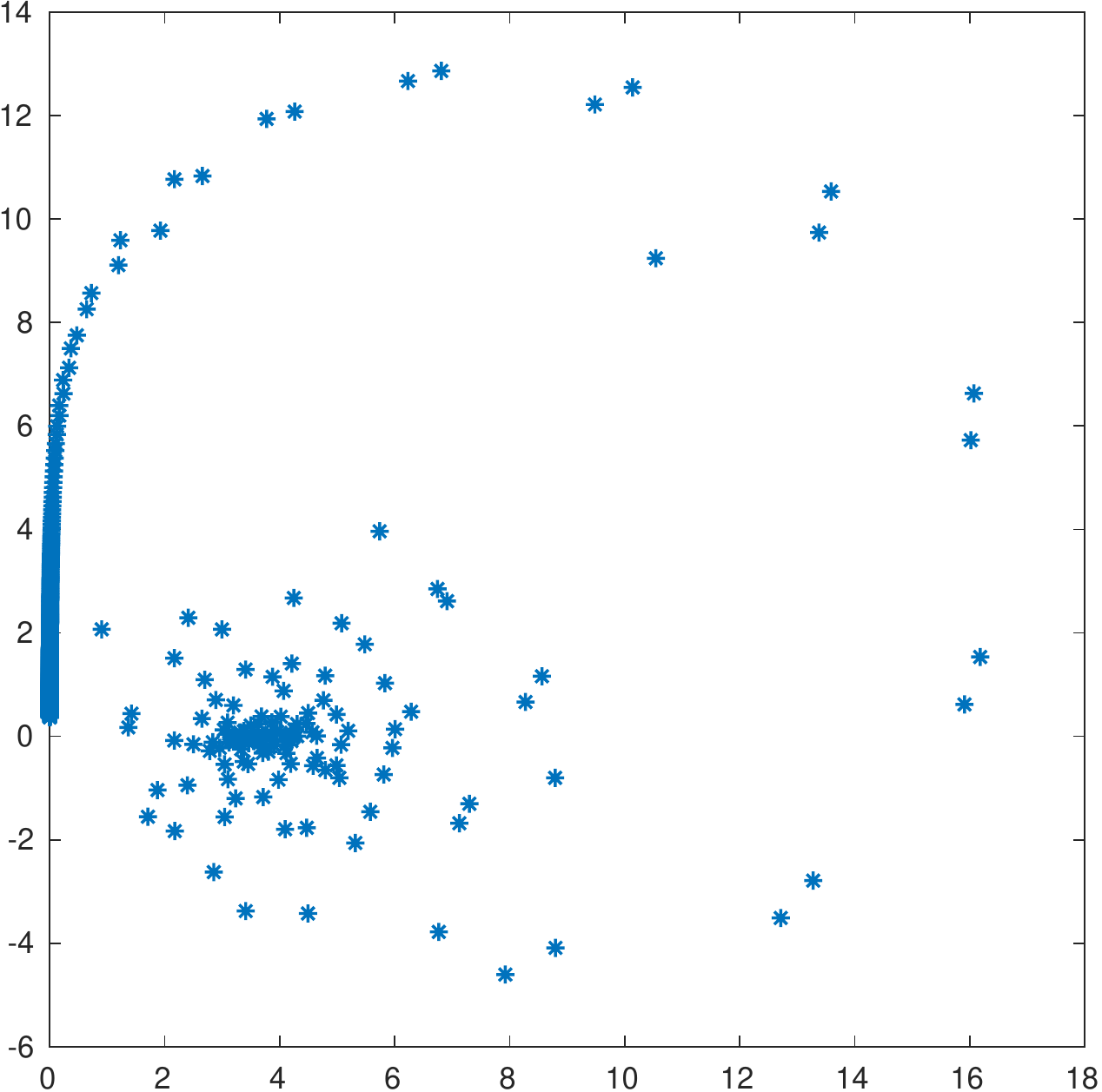} &
       \includegraphics[height=1.6in]{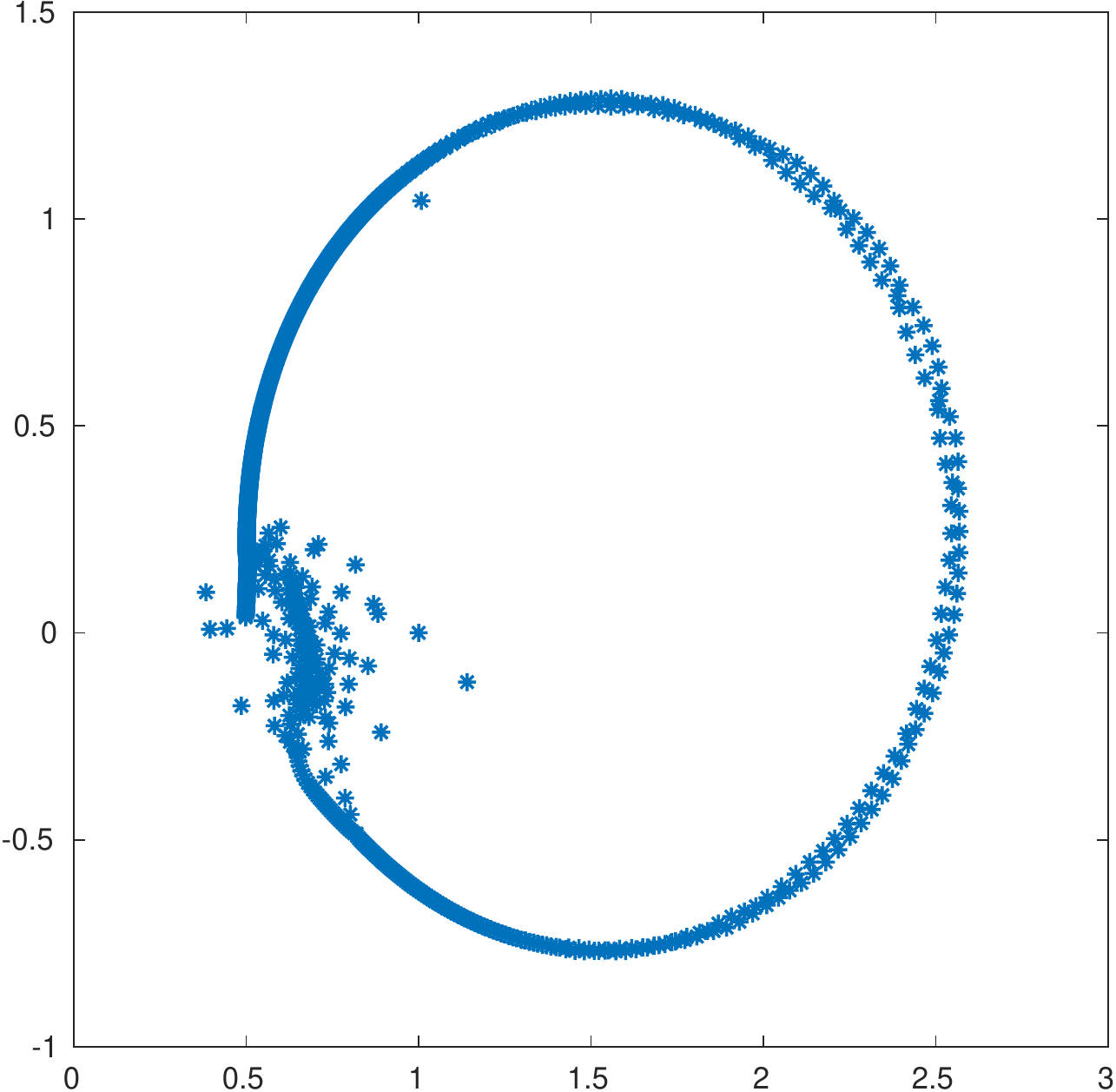} \\
       (a) & (b) & (c) \\
      \includegraphics[height=1.6in]{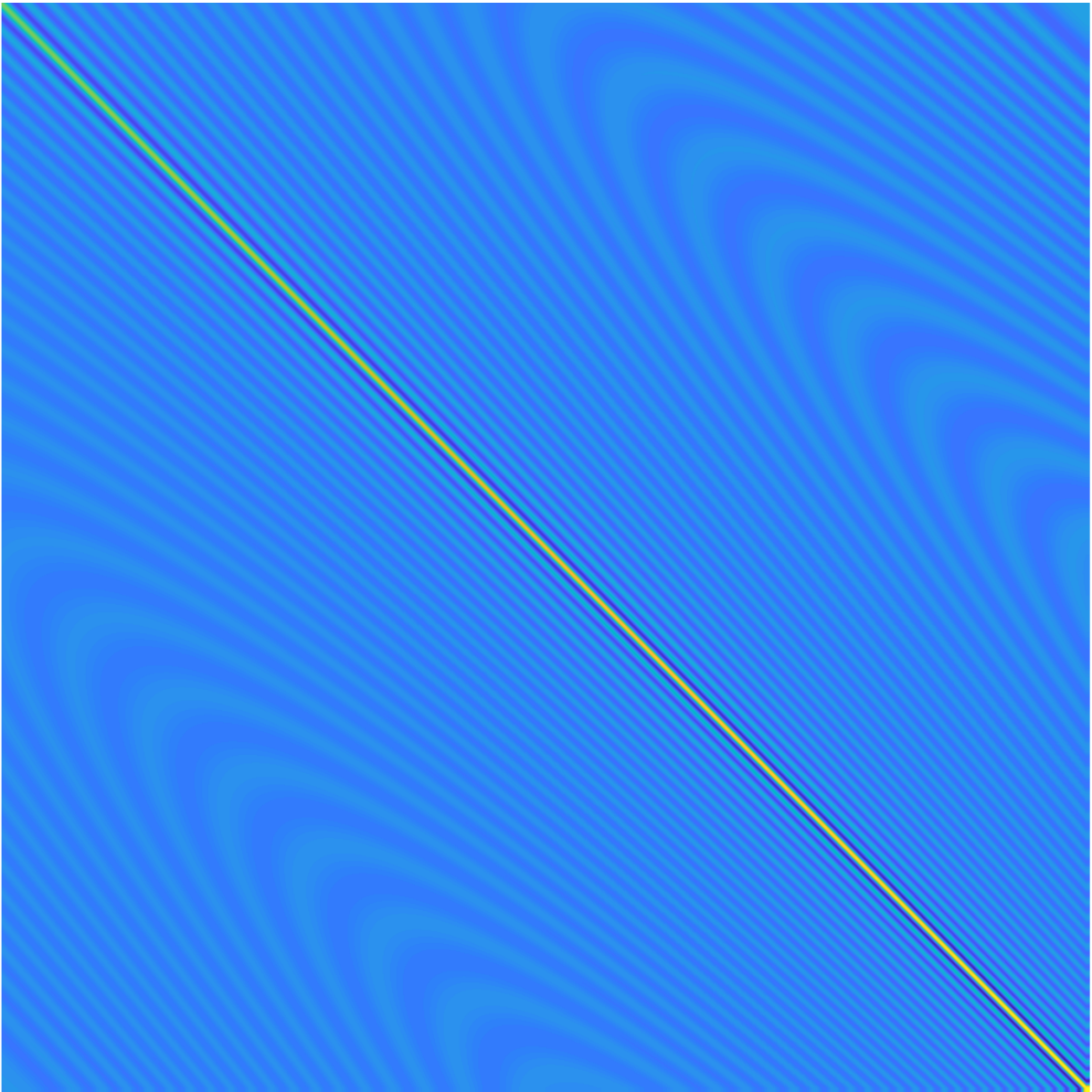} &  \includegraphics[height=1.6in]{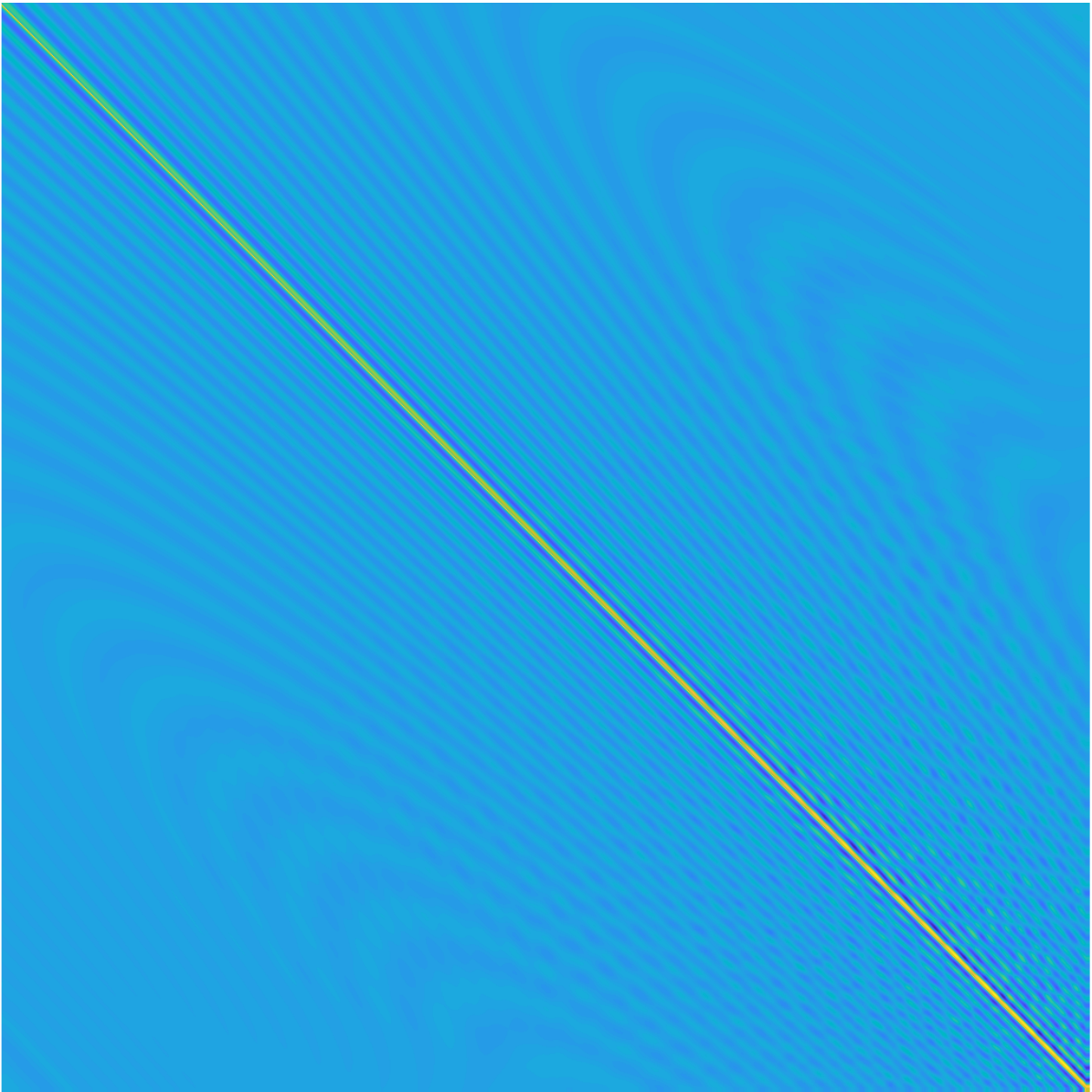} &
       \includegraphics[height=1.6in]{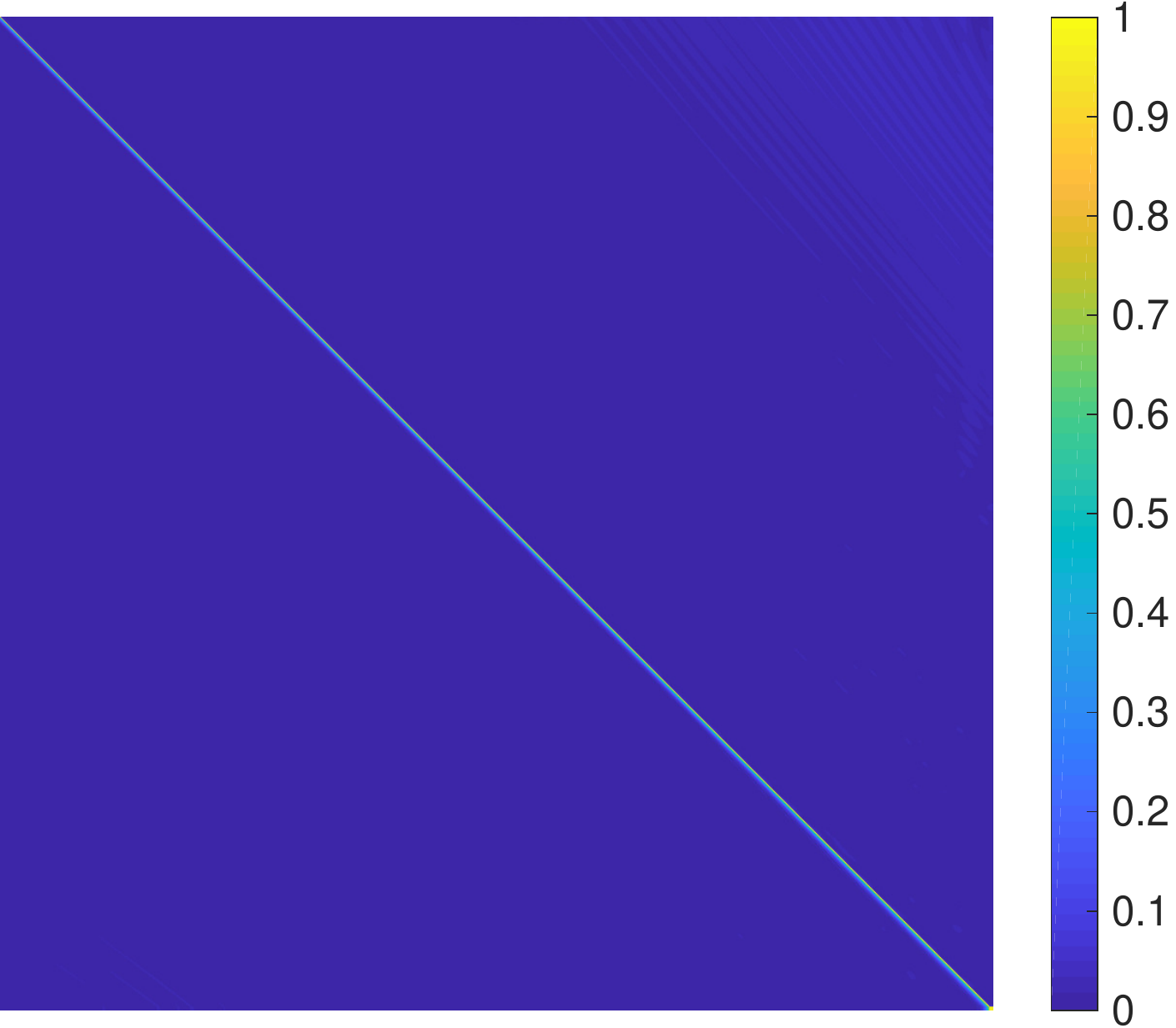} \\
       (d) & (e) & (f)
    \end{tabular}
  \end{center}
\caption{(a) A spiral scatterer. (b) The eigenvalue distribution of the impedance matrix $A$ corresponding to (a). (c) The eigenvalue distribution of $\tilde{L}^{-1}A\tilde{U}^{-1}$, the impedance matrix after preconditioning. (d) The real part of $A$. (e) The real part of $L+U-I$, where $L$ and $U$ are the LU factors of $A$, and $I$ is an identity matrix. (f) The magnitude of the entries of $\tilde{L}^{-1}A\tilde{U}^{-1}$. }
\label{fig:ex1}
\end{figure}

\begin{figure}[ht!]
  \begin{center}
    \begin{tabular}{ccc}
      \includegraphics[height=1.6in]{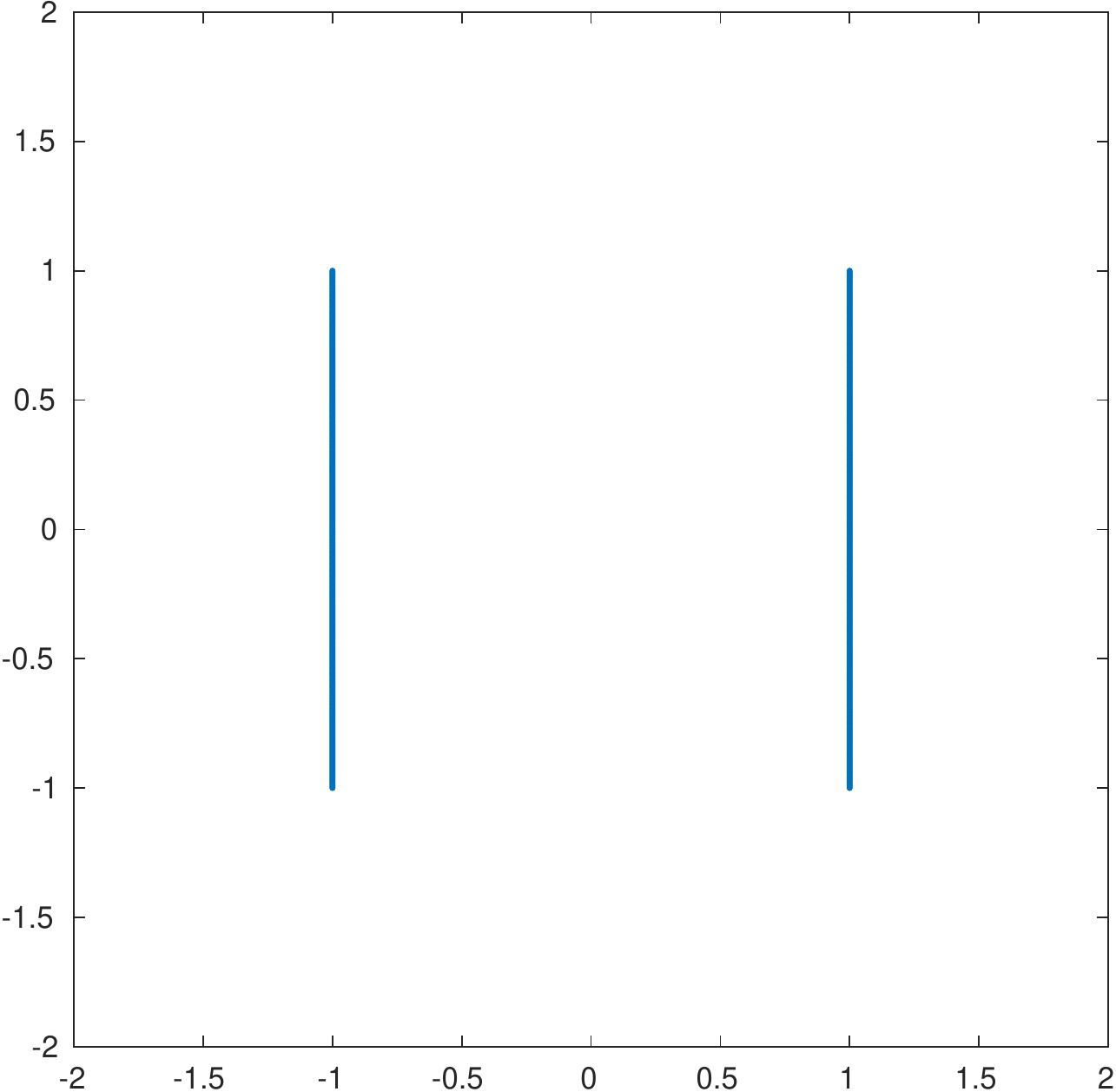} &  \includegraphics[height=1.6in]{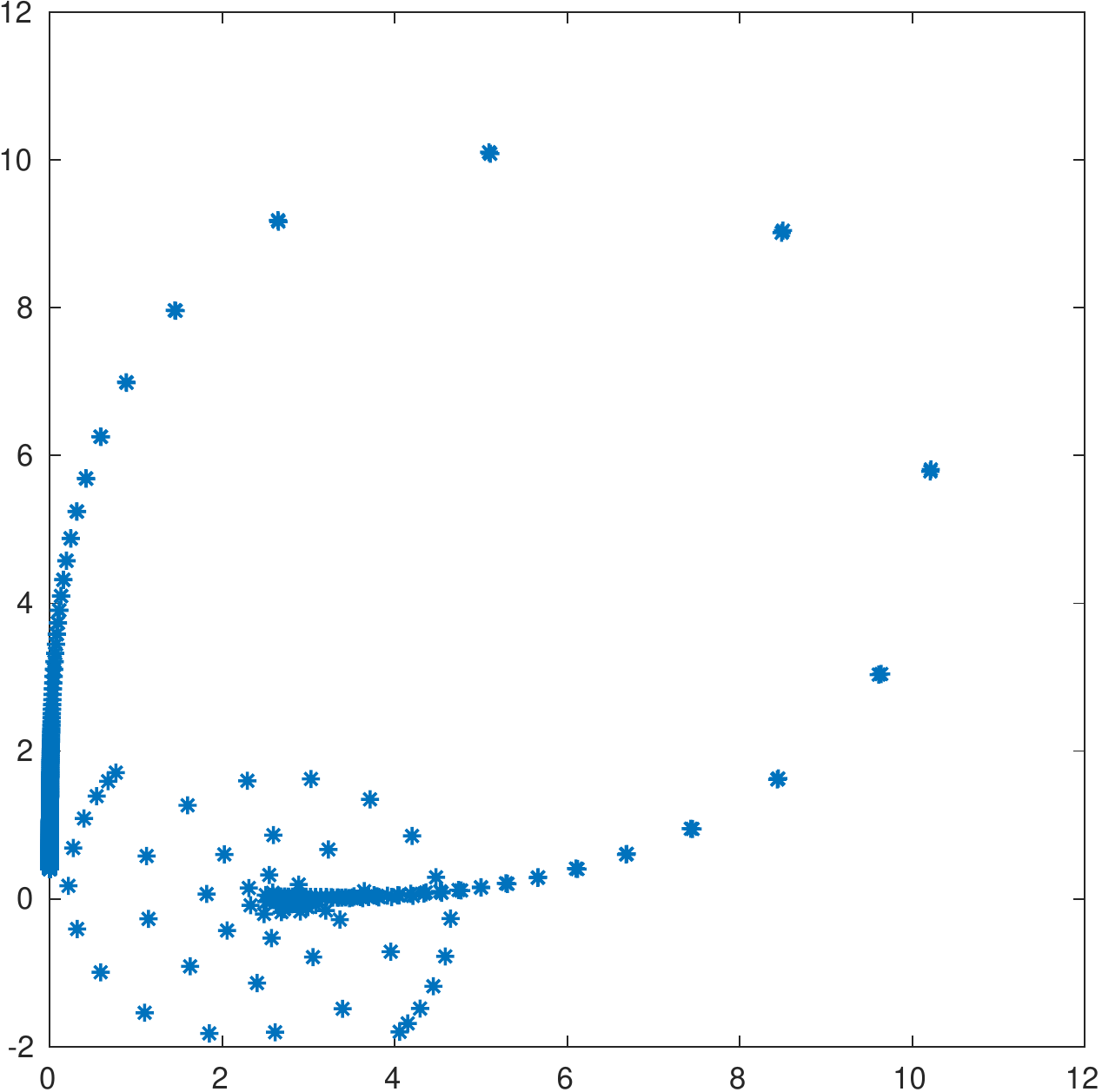} &
       \includegraphics[height=1.6in]{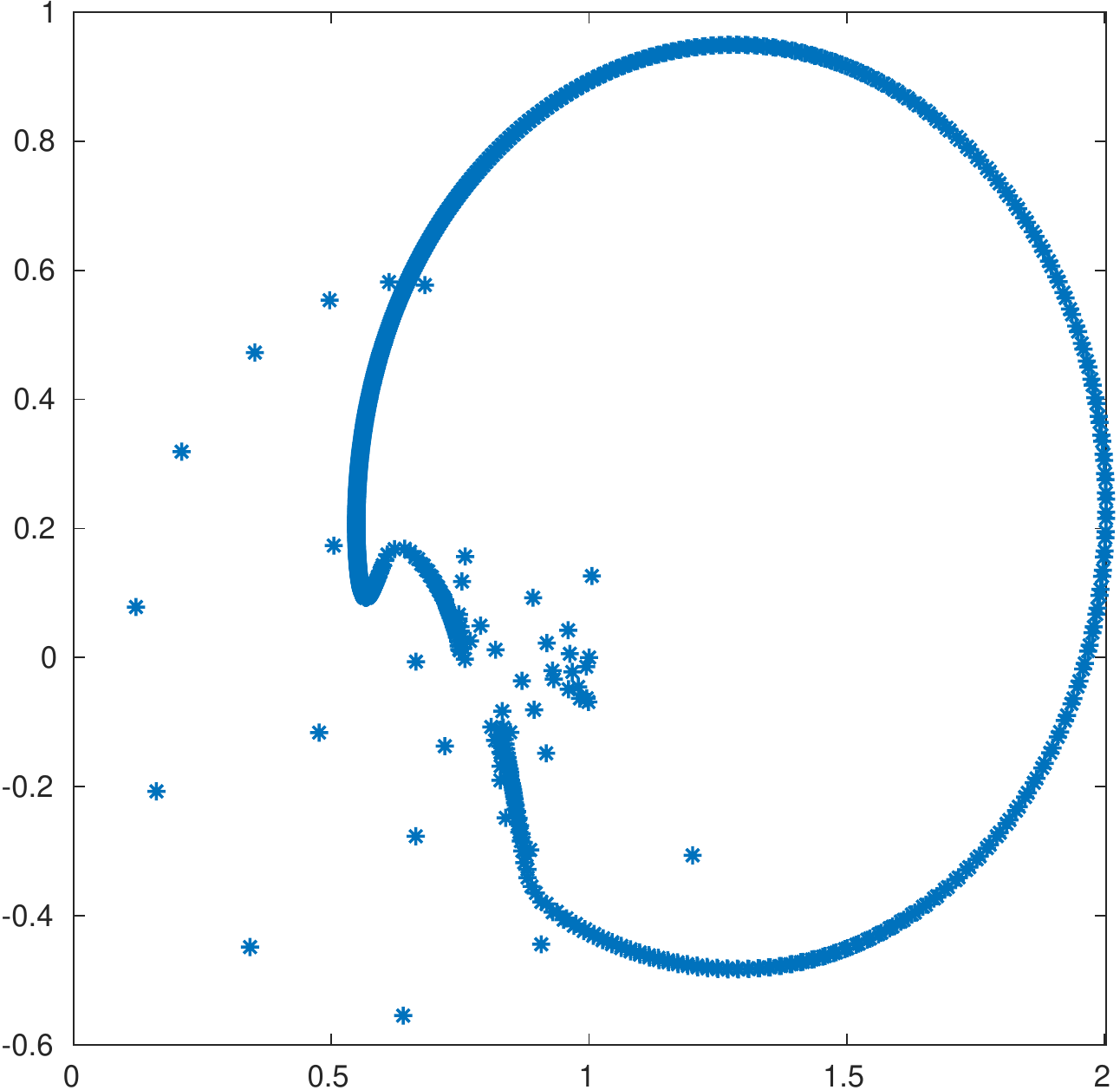} \\
       (a) & (b) & (c) \\
      \includegraphics[height=1.6in]{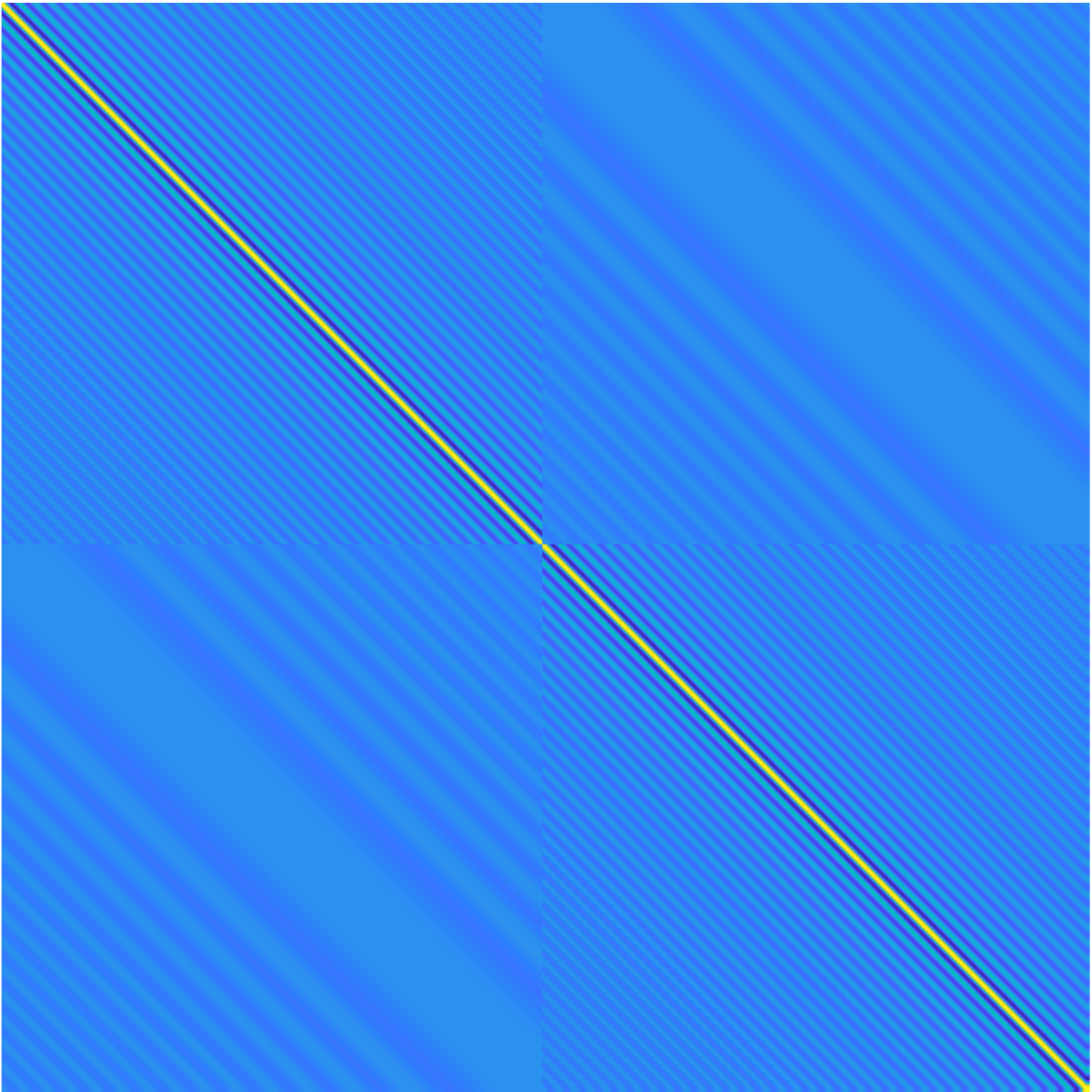} &  \includegraphics[height=1.6in]{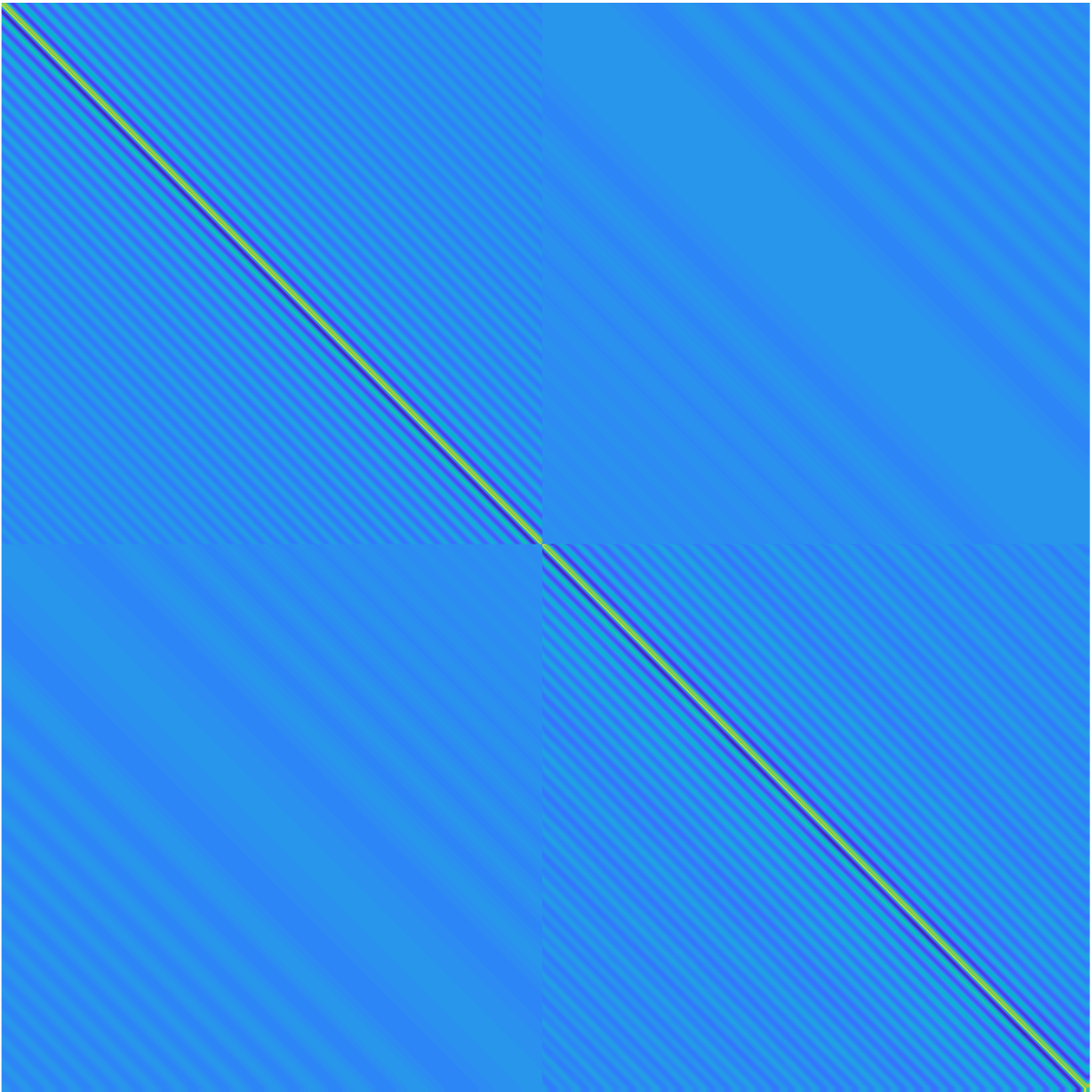} &
       \includegraphics[height=1.6in]{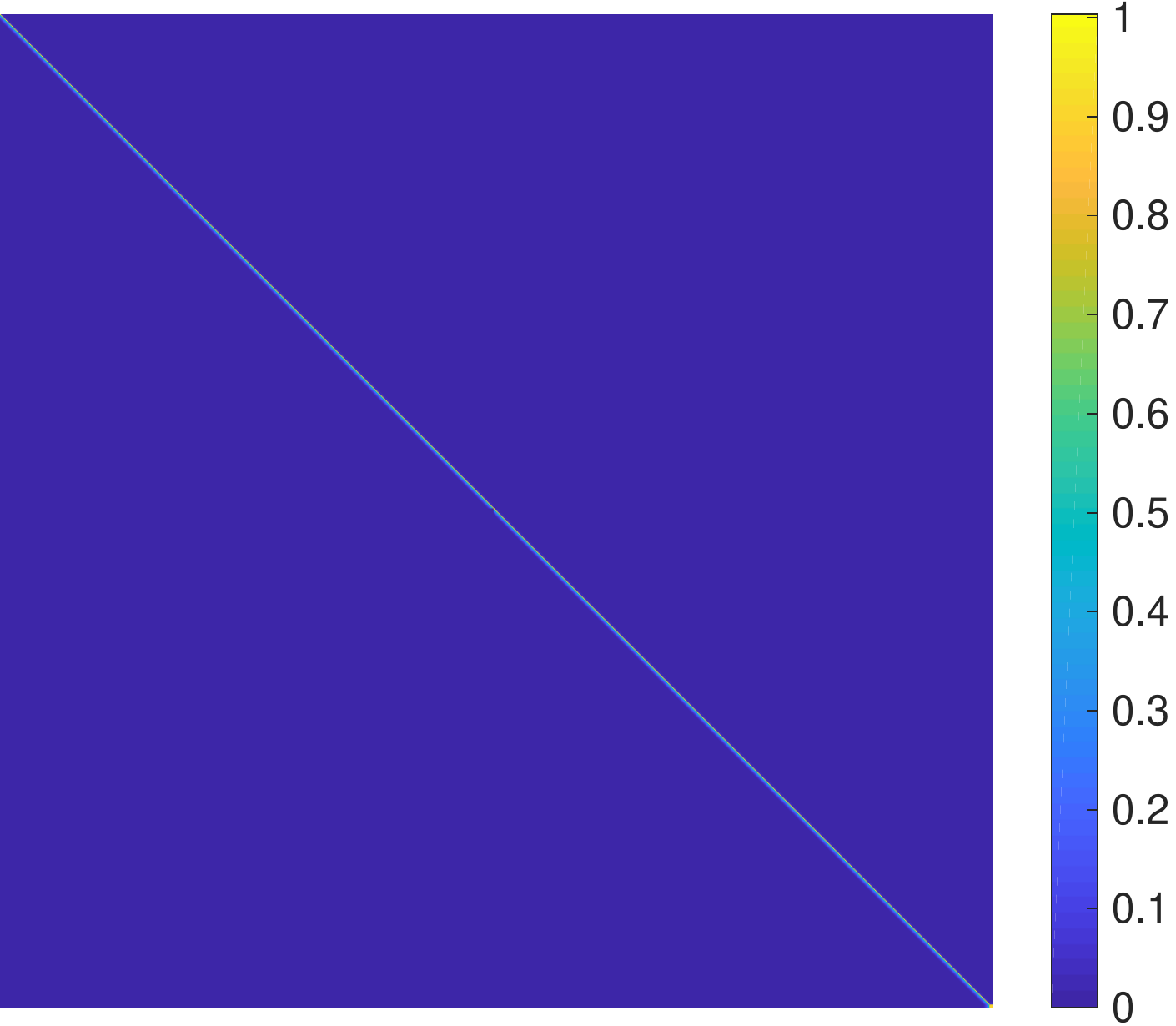} \\
       (d) & (e) & (f)
    \end{tabular}
  \end{center}
\caption{(a) Two parallel strips. (b) The eigenvalue distribution of the impedance matrix $A$ corresponding to (a). (c) The eigenvalue distribution of $\tilde{L}^{-1}A\tilde{U}^{-1}$, the impedance matrix after preconditioning. (d) The real part of $A$. (e) The real part of $L+U-I$, where $L$ and $U$ are the LU factors of $A$, and $I$ is an identity matrix. (f) The magnitude of the entries of $\tilde{L}^{-1}A\tilde{U}^{-1}$. }
\label{fig:ex2}
\end{figure}

The previous subsection has introduced $O(N\log^2N)$ algorithms for constructing and applying the H-IDBF of $A$ and forming $\tilde{L}$ and $\tilde{U}$ as a preconditioner. Next, we will introduce $O(N\log^2N)$ algorithms for applying the inverse of $\tilde{L}$ and $\tilde{U}$ to a vector in this subsection, i.e, solving triangular systems in a format of H-IDBFs. Similarly to solving a triangular system in a format of H-matrix, the triangular solver for H-IDBF can also be done recursively as in Algorithm \ref{alg:3} and \ref{alg:4}.

\vspace{0.25cm}
\begin{algorithm}[H]
 \KwData{A H-IDBF of $L\in\mathbb{C}^{N\times N}$ and a vector $b\in\mathbb{C}^N$.}
 \KwResult{A vector $x\in\mathbb{C}^{N}$ equal to $L^{-1}b$.}
 \eIf{$N\leq n_0$ }{
  Solve the triangualr system $Lx=b$ via standard LAPACK\;
   }{
   Partition the lower triangular system as 
  \[
  L=  \begin{pmatrix}
  F_{11} & 0\\
  F_{21} & F_{22}
  \end{pmatrix};
  \]
  Recursively apply Algorithm \ref{alg:3} with $F_{11}$ and $b(1:N/2)$ as the input to obtain $x(1:N/2)$ as the corresponding output\;
  Apply Algorithm \ref{alg:2} with $-F_{21}$ to obtain $b(N/2+1:N)\mathrel{-}=F_{21}x(1:N/2)$\;
  Recursively apply Algorithm \ref{alg:3} with $F_{22}$ and $b(N/2+1:N)$ as the input to obtain $x(N/2+1:N)$ as the corresponding output\;
  }
 \caption{An $O(N\log^2N)$ recursive algorithm for solving a lower triangular system $Lx=b$ when $L$ is stored in a format of H-IDBF. The notation ``$\mathrel{-}=$" in Line $5$ stands for the updating operator defined as: $a\mathrel{-}=b$ is equivalent to $a=a-b$ for any array $a$ and $b$.}
 \label{alg:3}
\end{algorithm}

\vspace{0.25cm}
\begin{algorithm}[H]
 \KwData{A H-IDBF of $U\in\mathbb{C}^{N\times N}$ and a vector $b\in\mathbb{C}^N$.}
 \KwResult{A vector $x\in\mathbb{C}^{N}$ equal to $U^{-1}b$.}
 \eIf{$N\leq n_0$ }{
     Solve the triangular system $Ux=b$ via standard LAPACK\;
   }{
     Partition the upper triangular system as
    \[
    U=  \begin{pmatrix}
    F_{11} & F_{12}\\
    0 & F_{22}
    \end{pmatrix};
    \]       
   Recursively apply Algorithm \ref{alg:4} with $F_{22}$ and $b(N/2+1:N)$ as the input to obtain $x(N/2+1:N)$ as the corresponding output\;
   Apply Algorithm \ref{alg:2} with $-F_{12}$ to obtain $b(1:N/2)\mathrel{-}=F_{12}x(1+N/2:N)$\;
   Recursively apply Algorithm \ref{alg:4} with $F_{11}$ and $b(1:N/2)$ as the input to obtain $x(1:N/2)$ as the corresponding output\;
  }
 \caption{An $O(N\log^2N)$ recursive algorithm for solving an upper triangular system $Ux=b$ when $U$ is stored in a format of H-IDBF. The notation ``$\mathrel{-}=$" in Line $5$ stands for the updating operator defined as: $a\mathrel{-}=b$ is equivalent to $a=a-b$ for any array $a$ and $b$.}
 \label{alg:4}
\end{algorithm}

\begin{figure}[ht!]
    \begin{center}
        \includegraphics[height=1in]{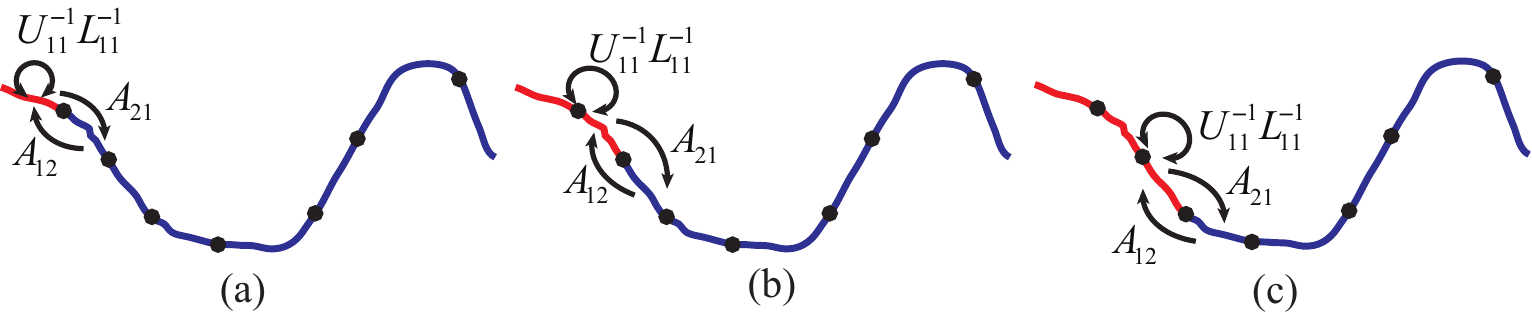}
    \end{center}
    \caption{{Eliminated and remaining unknowns at steps (a) $N/8$, (b) $N/4$ and (c) $3N/8$ of the LU factorization.} {Note that $A_{12}$ and $A_{21}$} are dense matrices whose dimensions corresponding to the number of unknowns in the red and blue colored curves.}
    \label{fig:illus}
\end{figure}

{
\subsection{Algorithm Analysis}
Here, we briefly analyze the effectiveness of the proposed LU preconditioner and the compressibility of off-diagonal blocks in the impedance matrix and its LU factors. 

\subsubsection{Oscillation patterns of $A$ and its LU factors}
\label{sec:analysis_lu}
{We observed that the oscillation pattern of $A$ resembles that of its LU factors. More specifically, $\tilde{L}\approx c_lL$ and $\tilde{U}\approx c_uU$, where $\tilde{L}$/$\tilde{U}$ are introduced in \eqref{eqn:Lt}/\eqref{eqn:Ut}, $L$ and $U$ are the LU factors of $A$, and $c_l$ and $c_u$ are scalar constants with $c_lc_u=1$. In this subsection, we provide some physical explanations of this observation assuming that} 1) the scatterer consists of smooth non-resonant surfaces; 2) the LU factorization follows a geometrical elimination order defined as follows. 

Assume that the scatterers consist of one simple contour with an analytic boundary. Let $\rho:[0,1)\rightarrow \mathbb{R}^2$ be the analytic map that maps $[0,1)$ on to the boundary  and discretize $[0,1)$ with a uniform grid $x_i=\frac{i-1}{N}$ for $i=1,\dots,N$. The geometrical order is the order that corresponds to the natural order of the discretized parameter $x_i$. When the scatterers consist of multiple simple contours, the ordering between contours is not studied here and {the optimal ordering} remains an open question.  

For simplicity, it is assumed that $A$ has been rescaled such that $A(i,i)\approx 1$. We want to show $\tilde{L}\approx L$ and $\tilde{U}\approx U$, respectively. Note that the LU factorization process at the step ${p}\in\{1,...,N\}$ can be rewritten as 
\begin{equation}
A=\begin{bmatrix}
A_{11} & A_{12}\\ 
A_{21} & A_{22}
\end{bmatrix}=\begin{bmatrix}
L_{11} & \\ 
A_{21}U_{11}^{-1} & I
\end{bmatrix}\begin{bmatrix}
I & \\ 
 & S_{22}
\end{bmatrix}\begin{bmatrix}
U_{11} & L_{11}^{-1}A_{12}\\ 
& I
\end{bmatrix},
\end{equation}
where $A_{11}=L_{11}U_{11}$ represents the completed LU factorization of the leading submatrix of sizes ${p\times p}$, and $S_{22}=A_{22}-A_{21}U_{11}^{-1}L_{11}^{-1}A_{12}$ denotes the updated Schur complement to be factorized in the next step. It suffices to show $L_{11}^{-1}A_{12}\approx A_{12}$, $A_{21}U_{11}^{-1}\approx A_{21}$, and $S_{22}\approx A_{22}$. As the factorization follows the geometrical order, let $c_1$ and $c_2$ denote the subscatterer corresponding to the eliminated (i.e., $A_{11}$) and remaining (i.e., $A_{22}$) rows. See the curves highlighted in red and blue in Figure \ref{fig:illus} (a)-(c) for illustrations. 

First note that $S_{22}$ represents the numerical Green's function that accounts for the curve $c_1$ as the background medium. As $c_1$ is smooth and $c_2$ resides on one side of $c_1$, the scattering component $A_{21}U_{11}^{-1}L_{11}^{-1}A_{12}$ does not dominate over the free-space interaction component $A_{22}$. {Physically, this means that the current induced on $c_1$ by sources on $c_2$, namely $U_{11}^{-1}L_{11}^{-1}A_{12}$, has weak Green's function interaction in the directions towards $c_2$.} Therefore we can use an aggressive approximation $S_{22}\approx A_{22}$.

We then show $L_{11}^{-1}A_{12}\approx A_{12}$ and $A_{21}U_{11}^{-1}\approx A_{21}$ by induction. Note that ${p}=1$ is trivial and assume that this holds true for step ${p}-1$, i.e., $L_{11}\approx \tilde{L}_{11}$ and $U_{11}\approx \tilde{U}_{11}$ at step ${p}$. 
Let $b_j$ denote the $j$th column of $A_{12}$, $j=1,...,N-{p}$, the vector $v=A_{11}^{-1}b_j$ is the current solution on $c_1$ residing in free space excited by a dipole source residing on $c_2$. {Note that $c_2$ is located on one side of $c_1$. For the linear system $A_{11}v=b_j$, we can define the forward operator $\tilde{U}_{11}$ and backward operator $\tilde{L}_{11}$ which only propagates the current away from and towards the source point, respectively.} As $c_1$ is smooth and non-resonant, {the linear system can be approximated as $\tilde{U}_{11}v\approx b_j$} by replacing the integral operator $A_{11}$ by the forward operator $\tilde{U}_{11}$ which only propagates the current away from the source point. {For example, using the forward operator, the current $v(i_1)$ does not contribute to the field $b_j(i_2)$ given that $i_1<i_2$ (i.e., $x_{i_1}$ represents a point that are further away from the source when compared to $x_{i_2}$). Note that the approximation $A_{11}v\approx\tilde{U}_{11}v$ does not guarantee that $A_{11}\approx\tilde{U}_{11}$. For more detailed explanation of this approximation, the readers are referred to Section III of \cite{475091}. From such approximation, it is easy to see that $b_j=A_{11}v=L_{11}U_{11}v\approx \tilde{L}_{11}\tilde{U}_{11}v\approx\tilde{L}_{11}b_j$. Note that the first approximate equality holds via induction.} In other words,  $\tilde{L}_{11}^{-1}b_j\approx b_j$ or equivalently ${L}_{11}^{-1}A_{12}\approx\tilde{L}_{11}^{-1}A_{12}\approx A_{12}$. Similarly, it can be shown that $A_{21}U_{11}^{-1}\approx A_{21}$. 

The analysis above explains the observation $\tilde{L}\approx L$,  $\tilde{U}\approx U$ and suggests that $\tilde{L}\tilde{U}$ can be used as an efficient preconditioner. We would like to highlight a few related works here. Our analysis is closely related to the fix-point iteration methods such as the forward-backward method {\cite{475091,496263}} or the symmetrical Gauss-Seidel iteration method \cite{1391151} for solving integral equations with a certain type of excitation vectors. When these methods are used as a preconditioner in a Krylov subspace iterative solver for arbitrary excitations, our physical explanations above show effectiveness of such preconditioner. Another related work is the sweeping preconditioner for differential equation solvers \cite{sweeping2011} which requires a geometrical elimination order such that the Schur complement represents well-compressible numerical Green's functions. Here we show that for surface integral equation solvers, a geometrical order can induce Schur complement that resembles the free-space Green's function for smooth, non-resonant scatterers. {Both algorithms rely on the fact that the geometry points corresponding to the remaining unknowns reside approximately on one side of the geometry points corresponding to the the already eliminated unknowns.}

\subsubsection{Butterfly compressibility of off-diagonal blocks}
\label{sec:analysis_rank}

Here we analyze the complementary low-rank property of the off-diagonal blocks of $A$, and it follows from the arguments in Section \ref{sec:analysis_lu} that the LU factors of $A$ are also butterfly compressible. Note that the compressibility of off-diagonal blocks resulting from strong-admissibility is well-studied in \cite{YingDirect,Butterfly1}, here we extend the compressibility analysis to the off-diagonal blocks from weak-admissibility. 

To simplify the analysis, we assume that the scatterer has an analytic boundary and its diameter and length are of order $1$. The parametrization of the boundary leads to an impedance matrix of size $N$ by $N$ and the frequency parameter ${\kappa}$ is of order $N$. Let $\rho:[0,1)\rightarrow \mathbb{R}^2$ be the analytic map that maps $[0,1)$ on to the boundary of the scatterer and discretize $[0,1)$ with a uniform grid $x_i=\frac{i-1}{N}$ for $i=1,\dots,N$. Let $\rho_i=\rho(x_i)$, then the impedance matrix $A$ is
\begin{align*}
 A_{ij} &=
  \begin{cases}
   {\frac{{\kappa}\eta_0w_j}{4}} H_0^{(2)}({\kappa}|\rho_i-\rho_j|),        & \text{if } i\neq j, \\
   {\frac{{\kappa}\eta_0w_i}{4}}\left[ 1-\mathrm{i} \frac{2}{\pi}\ln \left( \frac{\gamma {\kappa}{w_i}}{4e} \right)  \right],       & \text{otherwise}.
  \end{cases}
\end{align*}
It is sufficient to show the butterfly compressibility of the off-diagonal block for $1\leq j<\frac{N}{2}<i\leq N$ assuming $N$ is an even number. The analysis for other off-diagonal blocks is similar. By the theory of non-oscillatory phase functions for special functions \cite{watson1995treatise,HEITMAN20151,Bremer2019}, the Hankel function $H_0^{(2)}(x)$ for $x\in(0,\infty)$ admits a representation
\[
H_0^{(2)}(x)=\sqrt{\frac{2}{\pi x}}\frac{1}{\sqrt{\alpha'(x)}} e^{-i\alpha(x)},
\]
where $\alpha(x)$ is an analytic non-oscillatory phase function. {For more details about the definition, computation, and asymptotic expansion of the non-oscillatory phase function, the reader is referred to Section 2.5 of \cite{Bremer2019}, where the notation $\alpha(t)$ is used to represent the phase function as well. }
If we introduce a non-oscillatory matrix $B$ and a purely oscillatory matrix $C$ as follows,
\begin{align}\label{eqn:B}
 B_{ij} &=
  \begin{cases}
   \frac{{\kappa}\eta_0w_j}{4}   \sqrt{\frac{2}{\pi {\kappa}|\rho_i-\rho_j|}}\frac{1}{\sqrt{\alpha'({\kappa}|\rho_i-\rho_j|)}}     ,        & \text{if } i\neq j, \\
   {\frac{{\kappa}\eta_0w_i}{4}}\left[ 1-\mathrm{i} \frac{2}{\pi}\ln \left( \frac{\gamma {\kappa}{w_i}}{4e} \right)  \right],       & \text{otherwise},
  \end{cases}
\end{align}
and
\begin{align}\label{eqn:C}
 C_{ij} &=
  \begin{cases}
 e^{-i \alpha({\kappa}|\rho_i-\rho_j|)}   ,      & \text{if } i\neq j, \\
   1 ,      & \text{otherwise},
  \end{cases}
\end{align}
then $A$ is the Hadamard product of $B$ and $C$.

 \begin{figure}[ht!]
  \begin{center}
    \begin{tabular}{c}
      \includegraphics[height=1.6in]{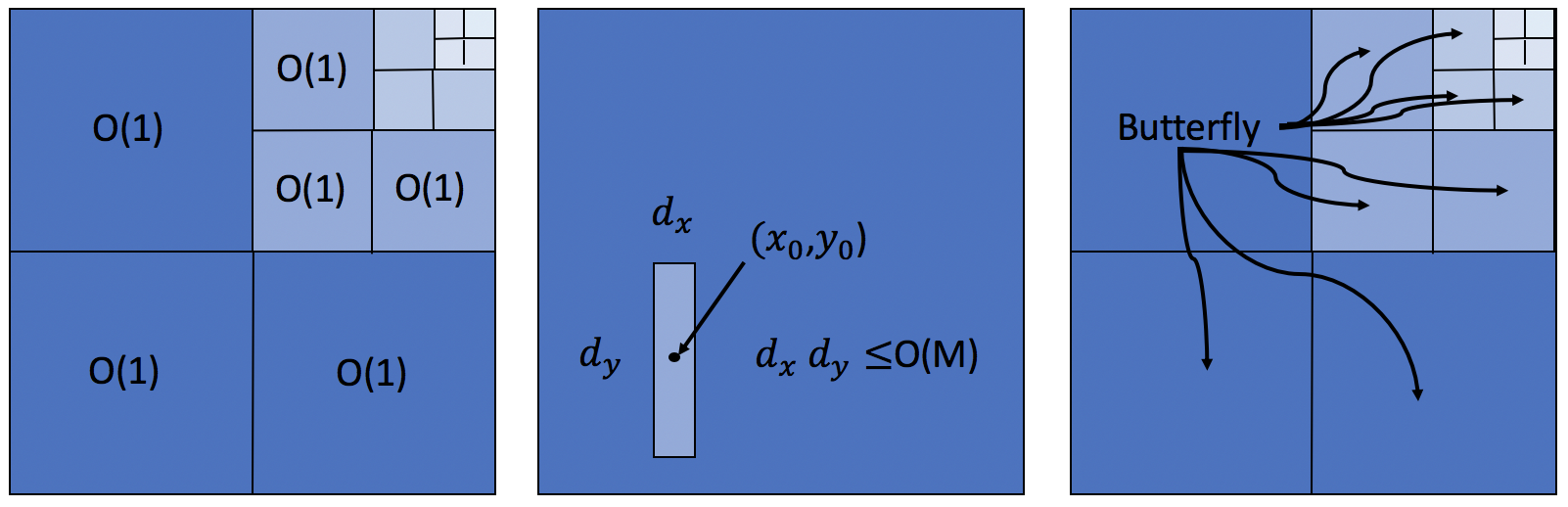} 
    \end{tabular}
  \end{center}
\caption{Left: the low-rank structure of the off-diagonal matrix of the non-oscillatory matrix $B$ for $1\leq j<\frac{N}{2}<i\leq N$. All submatrices visualized are rank $O(1)$. Middle: any contiguous submatrix of $C$ with $O(M)$ entries is rank $O(1)$ if the submatrix is away from the diagonal entries of $C$ with a distance at least $O(M)$ entries for $M=1,\dots,N$. Right: the low-rank structure of the off-diagonal matrix of the purely oscillatory matrix $C$ for $1\leq j<\frac{N}{2}<i\leq N$. All visualized submatrices are complementary low-rank matrices and hence butterfly compressible.  }
\label{fig:proof}
\end{figure}

Note that ${\kappa}|\rho_i-\rho_j|\geq O(1)$ for all $i\neq j$. Hence, $ M(x_i,x_j):=\sqrt{\frac{2}{\pi {\kappa}|\rho(x_i)-\rho(x_j)|}}\frac{1}{\sqrt{\alpha'({\kappa}|\rho(x_i)-\rho(x_j)|)}}    $ is analytic and non-oscillatory in $(x_i,x_j)$ leading to the following lemma.

\begin{lemma}\label{lem:1}
Suppose the map $\rho$ associated to the boundary of the scatterer is analytic and nearly isomatric, i.e., $c_1|x-y|\leq |\rho(x)-\rho(y)|\leq c_2|x-y|$ for some positive constants $c_1$ and $c_2$. Suppose $B$ of size $N$ by $N$ is defined in \eqref{eqn:B}, then the rank of the off-diagonal block of $B$ for $1\leq j<\frac{N}{2}<i\leq N$ is $O(\log(N))$ with a prefactor depending on a relative approximation error $\epsilon$ and independent of $N$.
\end{lemma}

\begin{proof}
Consider the range $x\in[x_0-\frac{d_x}{2},x_0+\frac{d_x}{2}]$ and  $y\in[y_0-\frac{d_y}{2},y_0+\frac{d_y}{2}]$ with $(x_0-\frac{d_x}{2})-(y_0+\frac{d_y}{2})\geq \frac{1}{N}$. Then as discussed previously, $ M(x,y)$ is analytic and non-oscillatory in $(x,y)$. Hence, as long as $d_x$ and $d_y$ are sufficiently small depending on $M(x,y)$ and are still of order $1$, a low-rank separation approximation for $M(x,y)$ with a relative error $\epsilon$ can be constructed via Taylor expansion or Chebyshev interpolation. The number of terms in the low-rank separation approximation depends only on $\epsilon$ and is independent of $N$. The conclusion is an immediate result of this fact.

For example, let us assume low-rank approximations are valid as long as $d_x\leq \frac{1}{4}$ and $d_y\leq\frac{1}{4}$. Then the off-diagonal block of $B$ is partitioned hierarchically as in Figure \ref{fig:proof} (left), the submatries of which have rank $O(1)$ independent of $N$. For an off-diagonal matrix of size $O(N)$ by $O(N)$, there are $O(\log(N))$ levels of submatrices. Hence, the total rank of the whole off-diagonal matrix in Figure \ref{fig:proof} (left) is $O(\log(N))$.
\end{proof}

The next lemma below analyzes the low-rank structure of the purely oscillatory matrix $C$.

\begin{lemma}\label{lem:2}
Suppose the map $\rho$ associated to the boundary of the scatterer is analytic and nearly isomatric, i.e., $c_1|x-y|\leq |\rho(x)-\rho(y)|\leq c_2|x-y|$ for some positive constants $c_1$ and $c_2$. Consider an off-diagonal contiguous submatrix of $C$ in \eqref{eqn:C} of size $d_x$ by $d_y$ with $d_xd_y\leq M$ and the submatrix is at least $O(M)$ entries away from the diagonal of $C$.  Then the submatrix is rank $O(1)$ with a prefactor depending on a relative approximation error $\epsilon$ and independent of $N$, when $M$ is larger than an $O(1)$ constant independent of $N$.
\end{lemma}
\begin{proof}
Without loss of generality, we consider
\begin{equation}\label{eqn:range}
x\in[x_0-\frac{d_x}{2N},x_0+\frac{d_x}{2N}]\qquad \text{and}\qquad y\in[y_0-\frac{d_y}{2N},y_0+\frac{d_y}{2N}]
\end{equation}
with 
\[
d_x d_y\leq M \qquad\text{and}\qquad (x_0-\frac{d_x}{2N})-(y_0+\frac{d_y}{2N})\geq O( \frac{M}{N}). 
\] 
Recall that $x$ and $y$ in $[0,1]$ are discretized with $N$ uniform grid points to form the matrix $C$. Hence, any submatrix in Lemma \ref{lem:2} is from the discretization of the conditions above. See Figure \ref{fig:proof} (middle) for an illustration of the submatrix of $C$ corresponding to the range in \eqref{eqn:range}.

By the Taylor expansion of $\alpha({\kappa}|\rho(x)-\rho(y)|)$ at the point $(x_0,y_0)$, there exists a function $\beta(x)$ and a function $\gamma(y)$ such that
\begin{equation}\label{eqn:sp}
e^{-i\alpha({\kappa}|\rho(x)-\rho(y)|)}=\beta(x) e^{-i \partial_{xy}\alpha({\kappa}|\rho(\bar{x})-\rho(\bar{y})|) (x-x_0)(y-y_0)}\gamma(y),
\end{equation}
with $\bar{x}\in [x_0-\frac{d_x}{2N},x_0+\frac{d_x}{2N}]$ and $\bar{y}\in [y_0-\frac{d_y}{2N},y_0+\frac{d_y}{2N}]$. Hence, the separability of the function $e^{-i\alpha({\kappa}|\rho(x)-\rho(y)|)}$ is equivalent to the separability of $e^{-i \partial_{xy}\alpha({\kappa}|\rho(\bar{x})-\rho(\bar{y})|) (x-x_0)(y-y_0)}$, i.e., to show Lemma \ref{lem:2} for the matrix $C$, it is sufficient to show the same property for the matrix $D$ defined below.
\begin{align}\label{eqn:D}
 D_{ij} &=
  \begin{cases}
 e^{-i \omega (x_i-x_0)(y_j-y_0)}    ,      & \text{if } i\neq j, \\
   1 ,      & \text{otherwise},
  \end{cases}
\end{align}
where $\omega=\partial_{xy}\alpha({\kappa}|\rho(\bar{x})-\rho(\bar{y})|) $ and  $x_i=y_i=\frac{i-1}{N}$ for $1\leq i\leq N$.

Let $\rho(x)=[\rho_1(x),\rho_2(x)]$. Then 
\[
\partial_{x}\alpha({\kappa}|\rho(x)-\rho(y)|) ={\kappa} \alpha'({\kappa}|\rho(x)-\rho(y)|)\frac{(\rho_1(x)-\rho_1(y))\rho_1'(x)+(\rho_2(x)-\rho_2(y))\rho_2'(x)}{|\rho(x)-\rho(y)|},
\]
and 
\begin{eqnarray*}
& &\partial_{xy}\alpha({\kappa}|\rho(x)-\rho(y)|)\\
&=&- {\kappa}^2 \alpha''({\kappa}|\rho(x)-\rho(y)|)\frac{(\rho_1(x)-\rho_1(y))\rho_1'(x)+(\rho_2(x)-\rho_2(y))\rho_2'(x)}{|\rho(x)-\rho(y)|}\\
& & \frac{(\rho_1(x)-\rho_1(y))\rho_1'(y)+(\rho_2(x)-\rho_2(y))\rho_2'(y)}{|\rho(x)-\rho(y)|}\\
&&-{\kappa} \alpha'({\kappa}|\rho(x)-\rho(y)|)\frac{\rho_1'(y)\rho_1'(x)+\rho_2'(y)\rho_2'(x)}{|\rho(x)-\rho(y)|}\\
&&+ {\kappa} \alpha'({\kappa}|\rho(x)-\rho(y)|) \frac{ (\rho_1(x)-\rho_1(y))\rho_1'(x)+(\rho_2(x)-\rho_2(y))\rho_2'(x)}{|\rho(x)-\rho(y)|^2}\\
&& \frac{(\rho_1(x)-\rho_1(y))\rho_1'(y)+(\rho_2(x)-\rho_2(y))\rho_2'(y)}{|\rho(x)-\rho(y)|}.
\end{eqnarray*}
Hence, for the range considered in \eqref{eqn:range}, we have $|\bar{x}-\bar{y}|\geq O(\frac{M}{N})$, then
\[
|\partial_{xy}\alpha({\kappa}|\rho(\bar{x})-\rho(\bar{y})|)|\leq O({\kappa}^2 \alpha''({\kappa}|\rho(\bar{x})-\rho(\bar{y})|) + \frac{Nk}{M}\alpha'(k|\rho(\bar{x})-\rho(\bar{y})|))),
\]
since $\rho$ is nearly isometric. Note that $\alpha(x)$ is non-oscillatory and analytic with the property (see Equations (57) in  \cite{Bremer2019}) that
\[
\lim_{x\rightarrow\infty} \alpha'(x)=1.
\]
Hence, there exists a positive constant $b_1$ independent of $N$ such that 
\[
\max_{x\in[a_1,\infty)}|\alpha'(x)|\leq b_1,
\]
where $a_1$ is a positive number such that $a_1\leq {\kappa}|\rho(x_i)-\rho(x_j)|$ as long as $i\neq j$, where $x_i=\frac{i-1}{N}$ for $i=1,\dots,N$, i.e., $a_1=O(1)$. Hence,
\[
|\frac{N{\kappa}}{M}\alpha'({\kappa}|\rho(\bar{x})-\rho(\bar{y})|))|\leq O(\frac{N{\kappa}}{M}).
\]
By the asymptotics of Bessel functions in Equations (63) in Section 2.5 of \cite{Bremer2019}, the following expansion is valid for all $x>0$:
\[
\alpha''(x)\sim \frac{-1}{4x^3}+\frac{25}{32x^5}+\frac{-3219}{512x^7}+\dots
\]
By a finite term truncation of the above expansion, there exist positive constants $c_2$ and $b_2$ such that $|\alpha''(x)|\leq \frac{c_2}{4x^3}$ for $x\in[b_2,\infty)$. Hence, as long as $M$ is sufficiently large, e.g, larger than $O(b_2)=O(1)$, ${\kappa}|\rho(\bar{x})-\rho(\bar{y})|\geq b_2$, and hence
\[
|{\kappa}^2 \alpha''({\kappa}|\rho(\bar{x})-\rho(\bar{y})|)|=O(\frac{1}{{\kappa}|\rho(\bar{x})-\rho(\bar{y})|^3})=O(\frac{N^2}{M^3}),
\]
since ${\kappa}=O(N)$ and $|\rho(\bar{x})-\rho(\bar{y})|=O(|\bar{x}-\bar{y}|)\geq O(\frac{M}{N})$. Note that \[|(x-x_0)(y-y_0)|=O(\frac{d_xd_y}{N^2})\leq O(\frac{M}{N^2}).\]
Hence,
\begin{eqnarray*}
&&|\omega(x-x_0)(y-y_0)|\\
&=&|\partial_{xy}\alpha({\kappa}|\rho(\bar{x})-\rho(\bar{y})|)(x-x_0)(y-y_0)|\\&\leq& O\left(\left(k^2 \alpha''({\kappa}|\rho(\bar{x})-\rho(\bar{y})|) + \frac{N{\kappa}}{M}\alpha'({\kappa}|\rho(\bar{x})-\rho(\bar{y})|)  \right)\frac{M}{N^2}\right)\\
&\leq & O(\frac{1}{M^2})+O(1)\\
&=&O(1),
\end{eqnarray*}
 which means that the submatrix of $D$ corresponding to the range in \eqref{eqn:range} is non-oscillatory and hence is rank $O(1)$ with a prefactor depending on a relative approximation error $\epsilon$ and independent of $N$. By \eqref{eqn:sp}, this conclusion is also true for the matrix $C$ and hence, we have completed the proof of Lemma \ref{lem:2}.

\end{proof}

An immediate result of Lemma \ref{lem:2} is that the off-diagonal block of $C$ is a hierarchically complementary low-rank matrix, e.g., see Figure \ref{fig:proof} (right) for an illustration. This conclusion is also true for the impedance matrix $A$.  Theoretically, the fast matvec of $A$ can be performed via hierarchically decomposition $A$ as in Figure \ref{fig:proof} (right) and applying the IDBF for each submatrix. This leads to a fast matvec with a complexity $O(N\log^2(N))$. Numerically, a more convenient way is to directly compress the whole off-diagonal block with IDBF, which also leads to a fast matvec with  a numerical complexity $O(N\log^2(N))$.

   }

\section{Numerical results}
\label{sec:results}
This section demonstrates the
efficiency and accuracy of the proposed preconditioner via its application to wide classes of open surfaces: a semicircle, a corrugated corner reflector, a spiral line, two parallel strips, a cup-shaped cavity, and {an array of open arcs}. In all examples, the surfaces are discretized with approximately 20 pulse basis functions per wavelength. The leaf sizes in the dyadic trees of each butterfly-compressed block are set to approximately $n_0=200$. The compression tolerance in linear scaling ID is set to $\epsilon$=1.0E-4, and the oversampling parameter $t$ in the linear scaling ID is set to $1$. For the aforementioned surfaces, the maximum butterfly ranks k among all blocks of the impedance matrices are respectively 7, 11, 10, 7, 14, 13, which are {almost} independent of the matrix sizes $N$. The matrix entries are scaled with a scalar such that the largest diagonal entry has unit magnitude. A transpose-free quasi-minimal residual (TFQMR) iterative solver is utilized with a convergence tolerance 1.0E-5. All experiments are performed on the Cori Haswell machine at NERSC, which is a Cray XC40 system and consists of 2388 dual-socket nodes with Intel Xeon E5-2698v3 processors running 16 cores per socket. The nodes are configured with 128\GB of DDR4 memory at 2133\MHz.          


 First, the accuracy of the proposed preconditioner is demonstrated by changing the matrix size $N$ from 5000 to 5,000,000. Let $x^{t}$ be a randomly generated $N\times1$ vector, the right hand side is generated as $b=Ax^{t}$. Note that $A$ is replaced by its H-IDBF compression $F$ when $N>10000$. Let $x^a$ denote the solution vector computed from the preconditioned linear system. The solution error is defined as $\epsilon^a=\left\lVert x^a-x^{t}\right\rVert_2/\left\lVert x^{t}\right\rVert_2$. From Table \ref{tab:error}, reasonably good accuracy has been observed for all classes of surfaces. However, a slight degradation in accuracy when $N$ increases is also observed.   

 Next, the computational efficiency is demonstrated by investigating the {computation time and memory} as matrix size $N$ increases. The construction time, iterative solution time, and overall memory are plotted in Figure \ref{fig:timemem} for all test surfaces. It can be easily validated that all three quantities scale as $O(N\log^2N)$ as predicted. It is worth mentioning that for the semicircle example, the proposed preconditioner permits rapid solution for $N=100$ million, which is very competitive to the state-of-the-art MLFMA-based iterative solvers. We also observe that the solution time (i.e., application and triangular solve of H-IDBF-LU in the iterative solver) dominates the overall computation time especially for surfaces that requires constant but a high number of iterations. 
 
 Finally, we compare the iteration counts of the proposed preconditioner and an iterative solver without any preconditioner (see Figure \ref{fig:iter}). {First, we investigate the dependence on $N$ for different surface shapes.} The iteration counts required by the iterative solver without preconditioner grow rapidly for all surfaces (dashed lines in Figure \ref{fig:iter}(a)). In contrast, for surfaces including the semicircle, spiral lines, and corrugated corners, the iteration counts using the proposed preconditioner stay as small constant (typically less than 30); for the other surfaces, the observed iteration counts scale at most as $O(\log^2N)$. It's worth mentioning that for complicated geometries such as highly resonant cavities or closed surfaces, the iteration counts can grow much faster. {To see this, we then investigate the dependence on the degree of resonance  or multiple reflections using an open arc with varying angles (Figure \ref{fig:iter}(b)), an Archimedean spiral with varying rotation angles (Figure \ref{fig:iter}(c)), (d) a 1-D array of open arcs with different element counts (Figure \ref{fig:iter}(d)). For the arc and spiral, the surface supports a higher degree of resonance as the angle increases (especially beyond $2\pi$), therefore the proposed preconditioner becomes less effective but still shows significant improvement compared to the iterative solver without any preconditioner. For the array of arcs, the scatterers support a dominant direction of propagation (along the direction of repetition) as the element increases, therefore the proposed preconditioner becomes very effective.}

\begin{figure}[ht!]
    \begin{center}
        \begin{tabular}{cc}
            \includegraphics[height=2in]{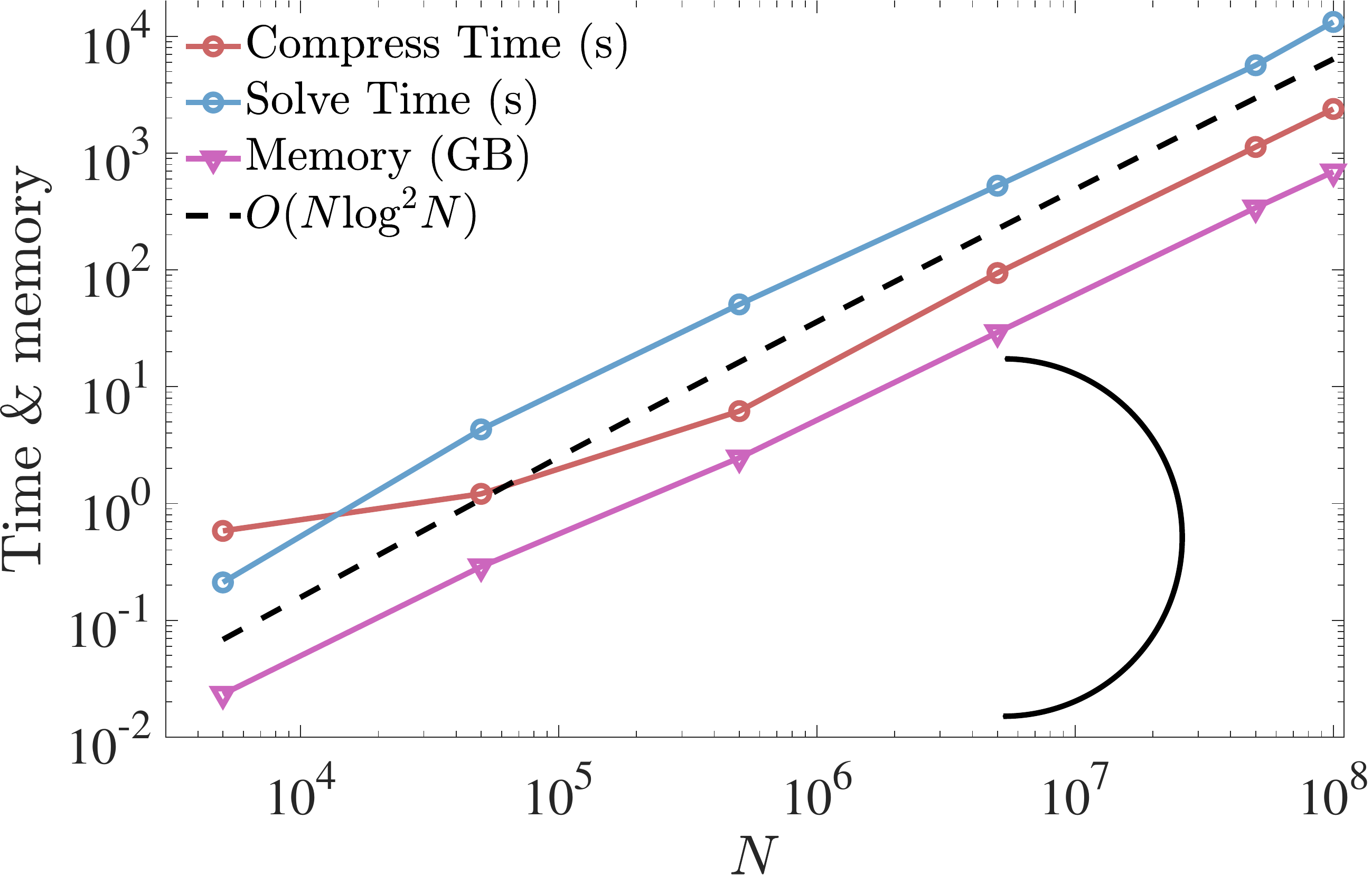} &  \includegraphics[height=2in]{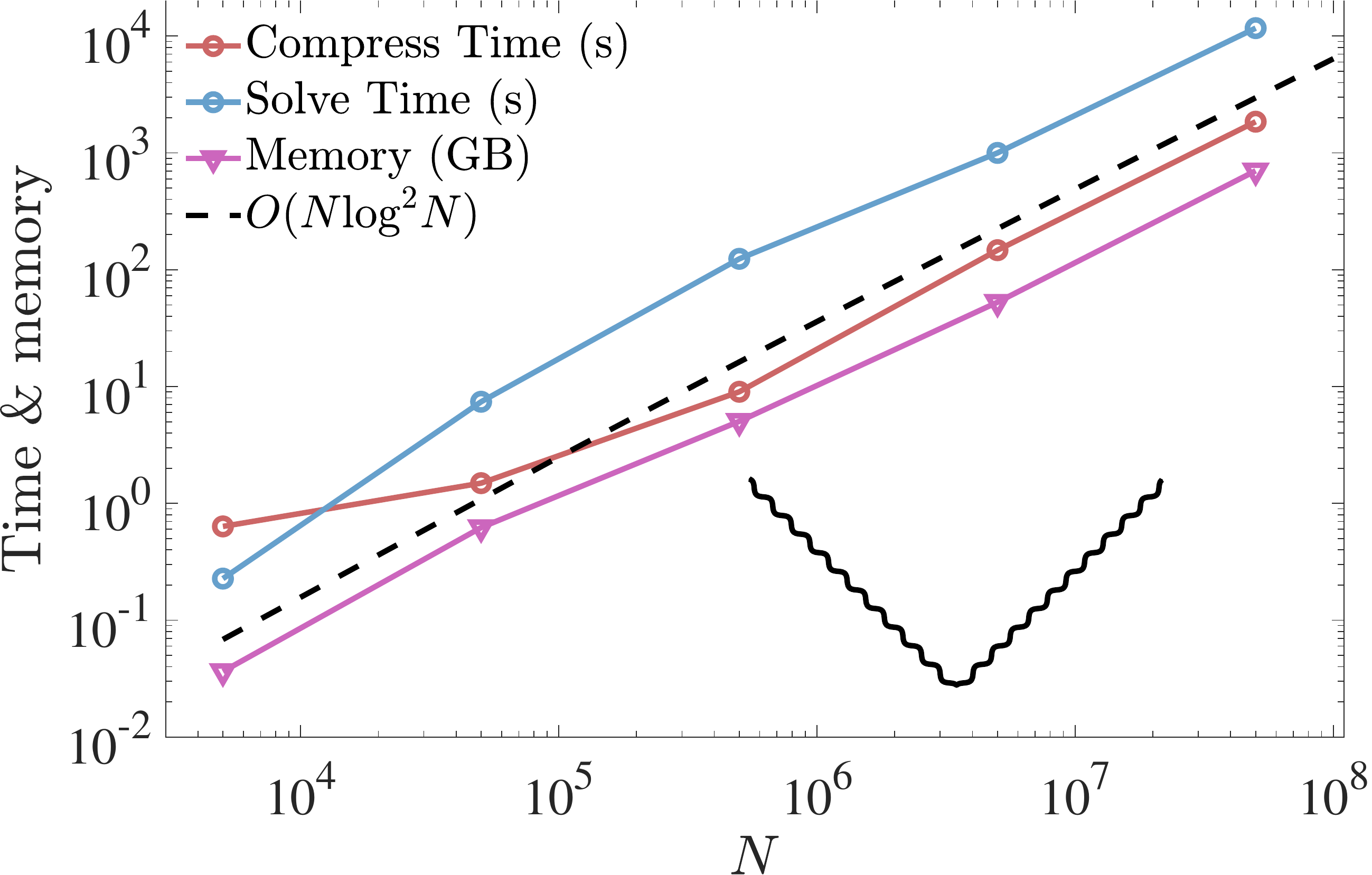} \\
            (a) & (b) \\
            \includegraphics[height=2in]{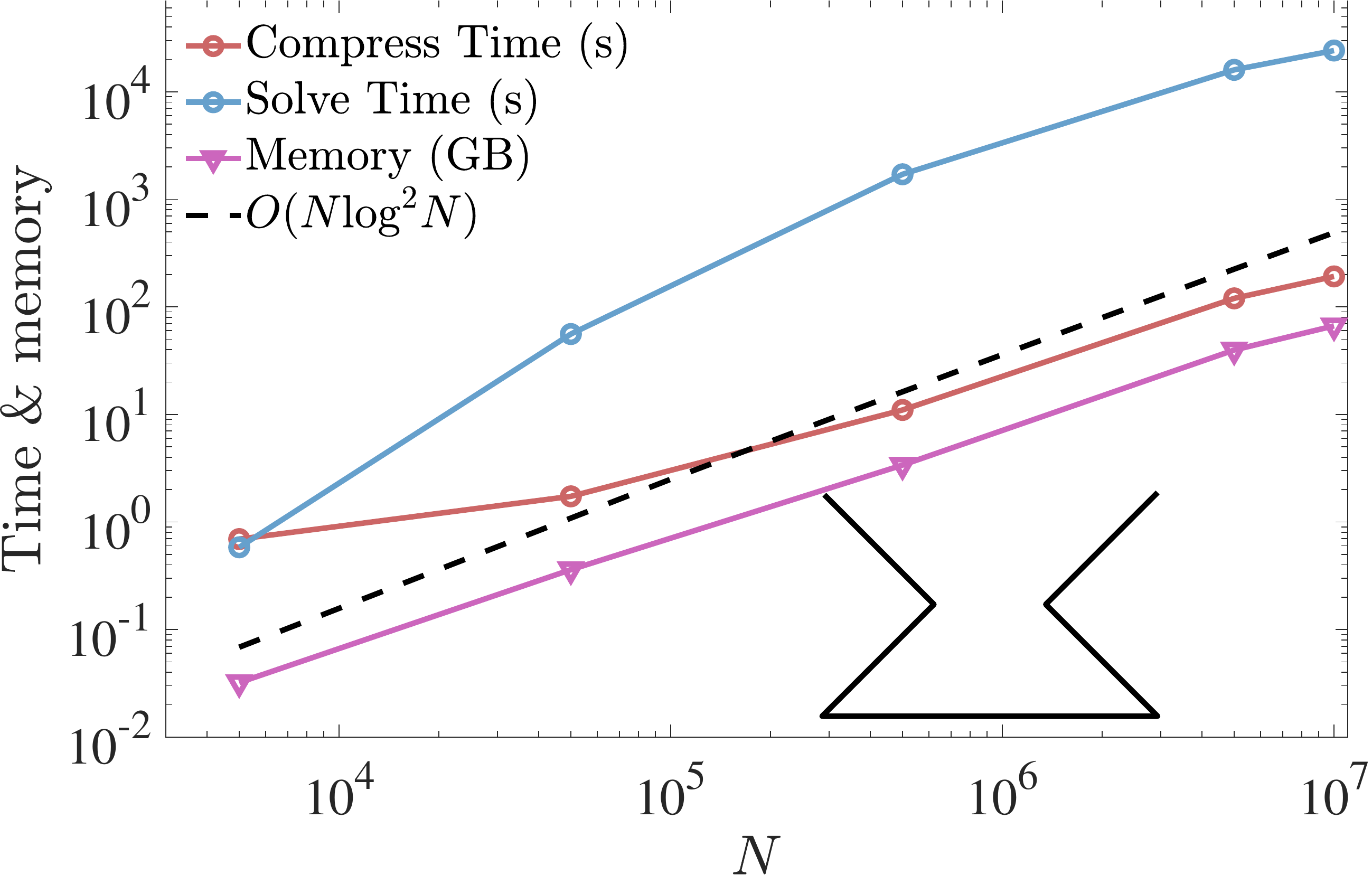} &  \includegraphics[height=2in]{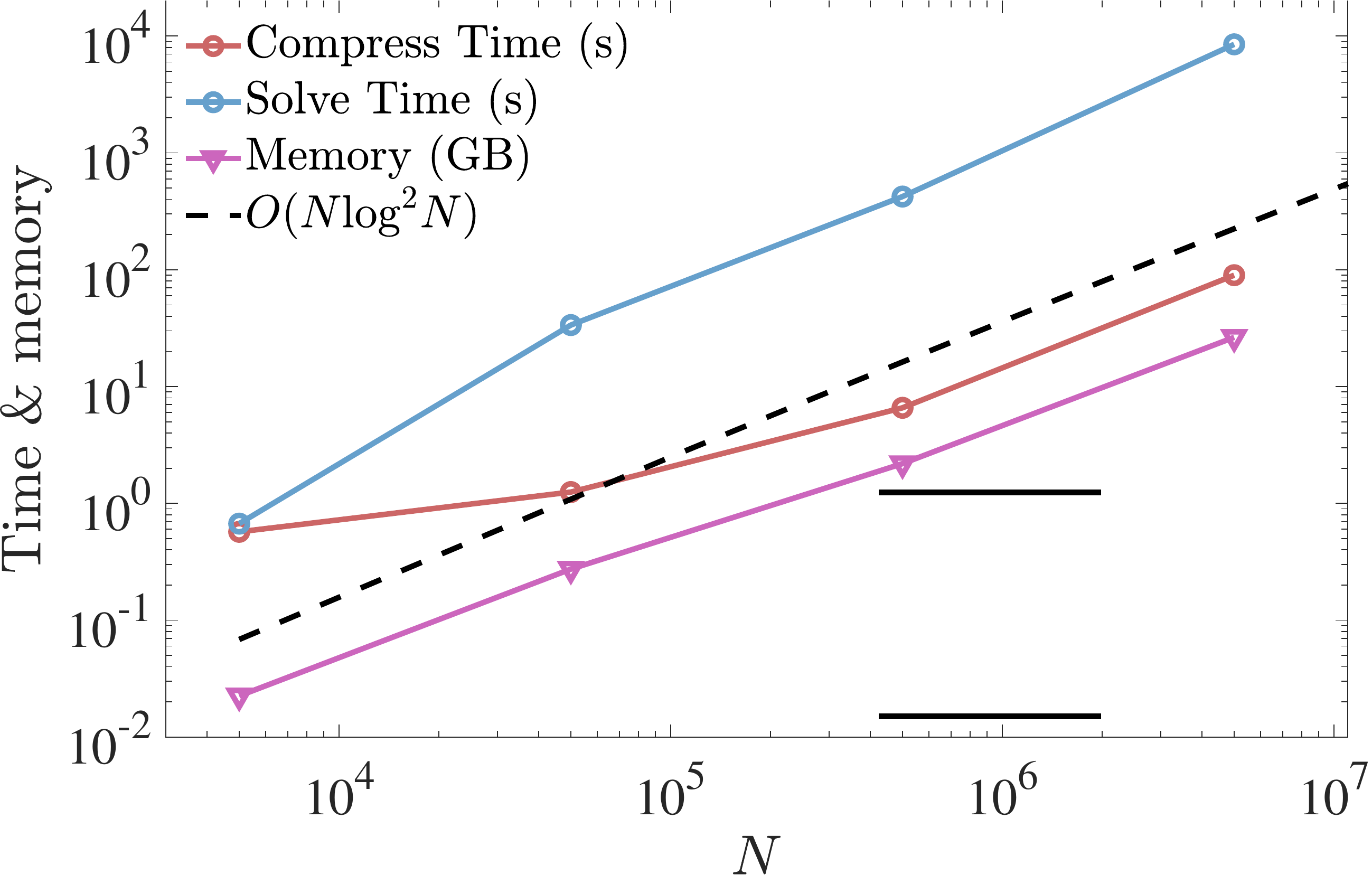} \\
            (c) & (d) \\    
            \includegraphics[height=2in]{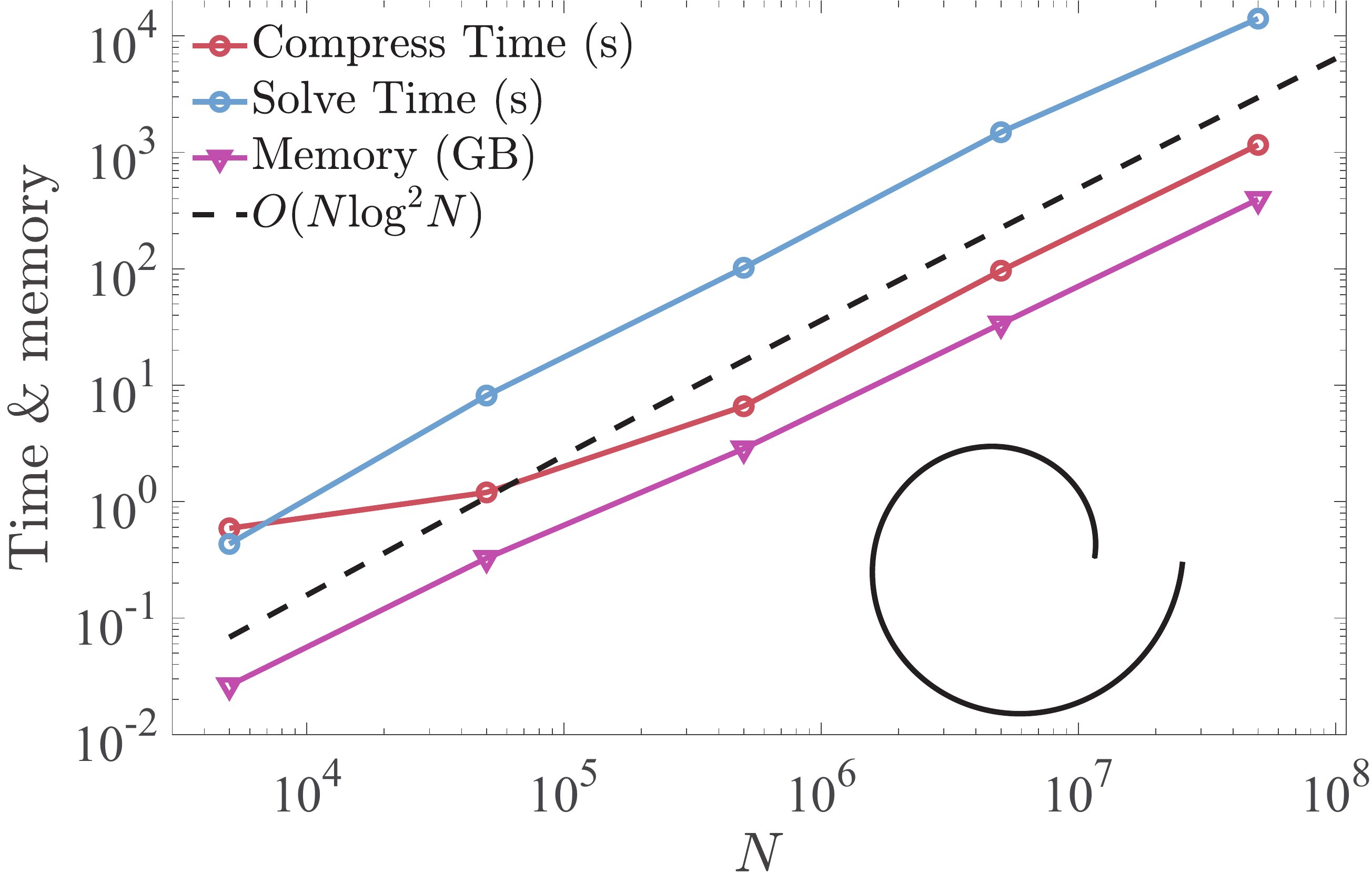} &  \includegraphics[height=2in]{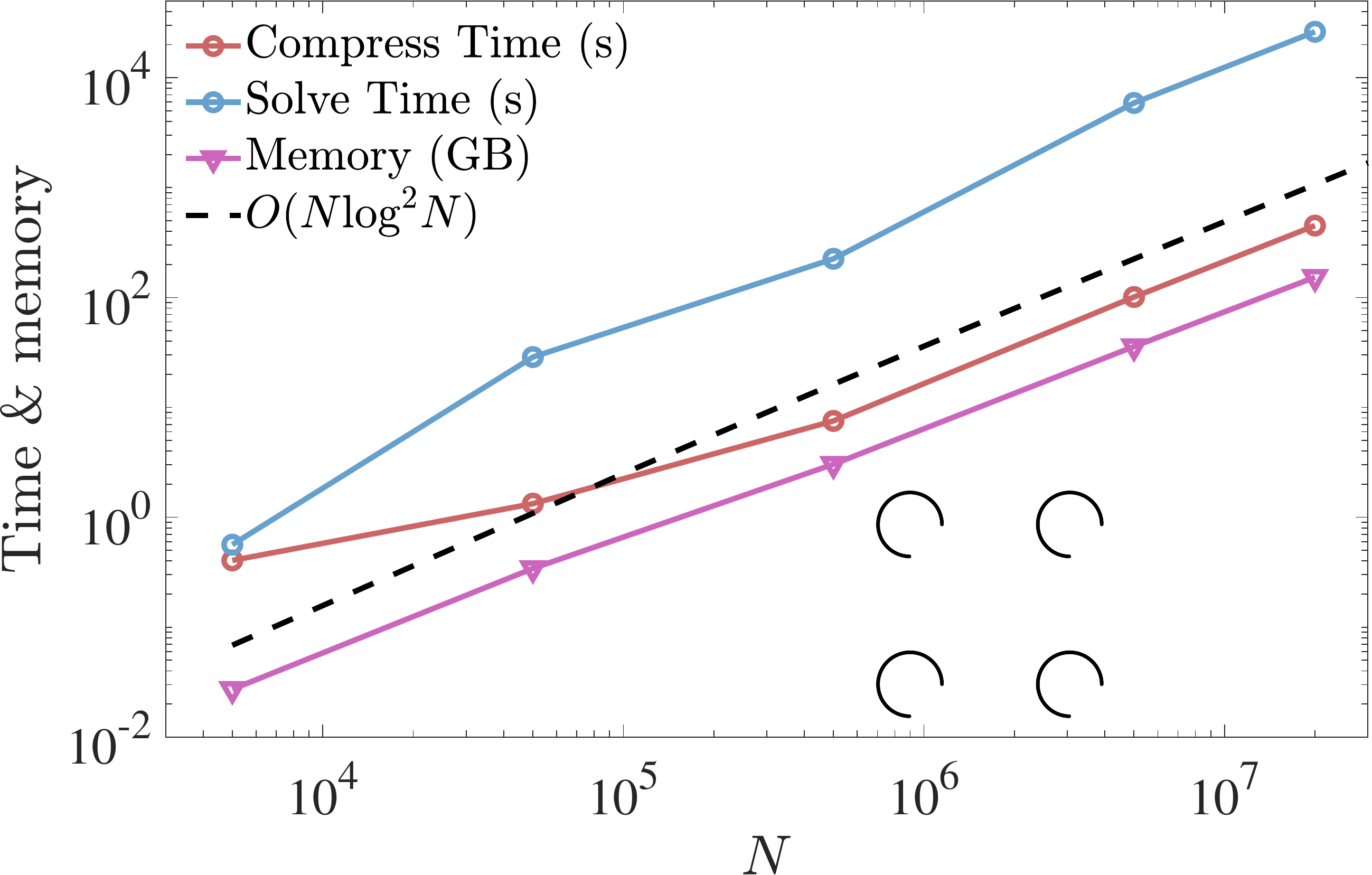} \\
            (e) & {(f)} 
        \end{tabular}
    \end{center}
    \caption{Computation time and memory of the proposed preconditioner for (a) a semicircle, (b) a corrugated corner, (c) a cup-shaped cavity, (d) two parallel strips, (e) a spiral line, and (f) {an array of 2$\times$2 arcs.} }
    \label{fig:timemem}
\end{figure}

\begin{figure}[ht!]
    \begin{center}
        \begin{tabular}{cc}
            \includegraphics[height=2in]{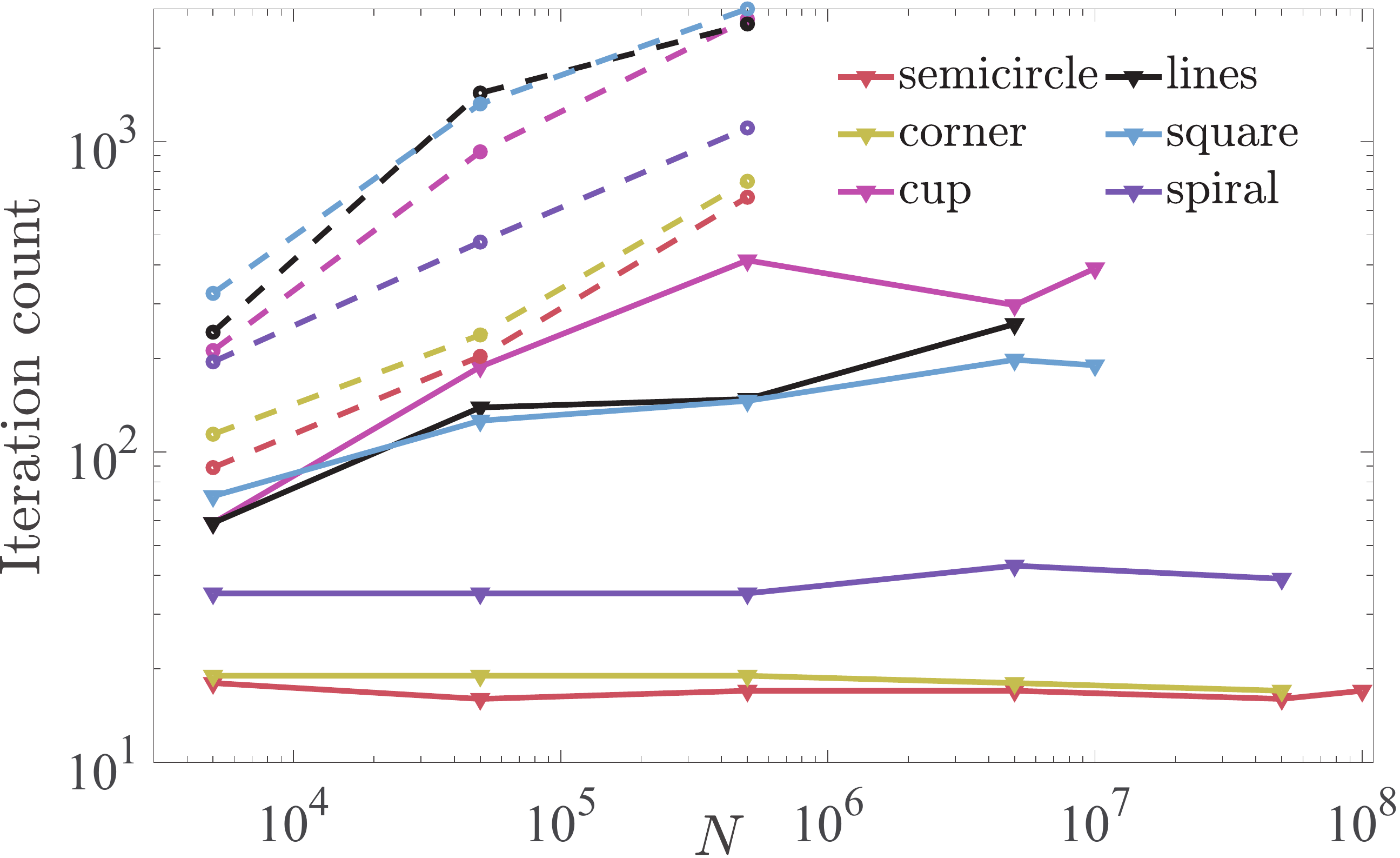} &  \includegraphics[height=2in]{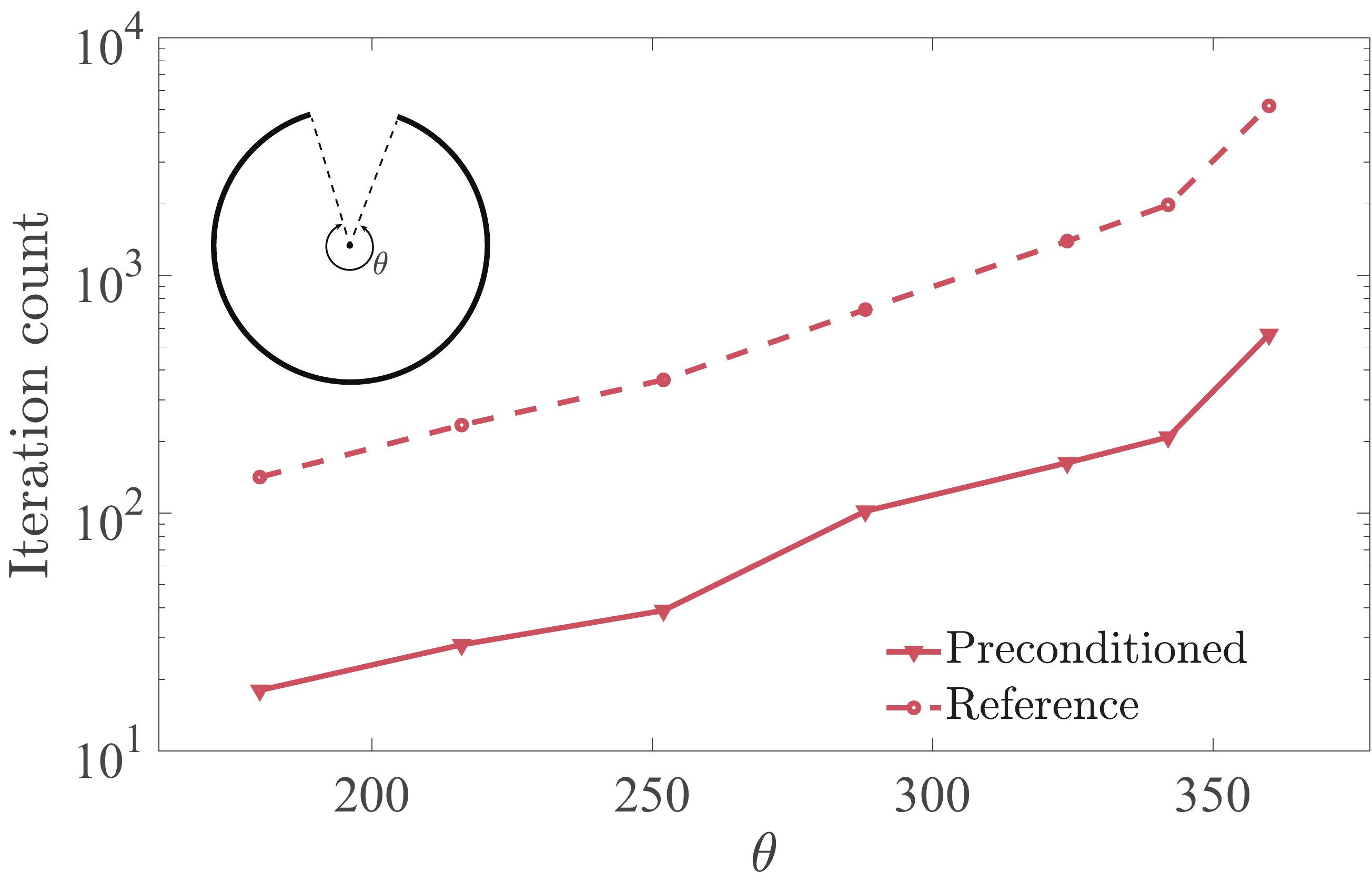} \\
            (a) & {(b)} \\
            \includegraphics[height=2in]{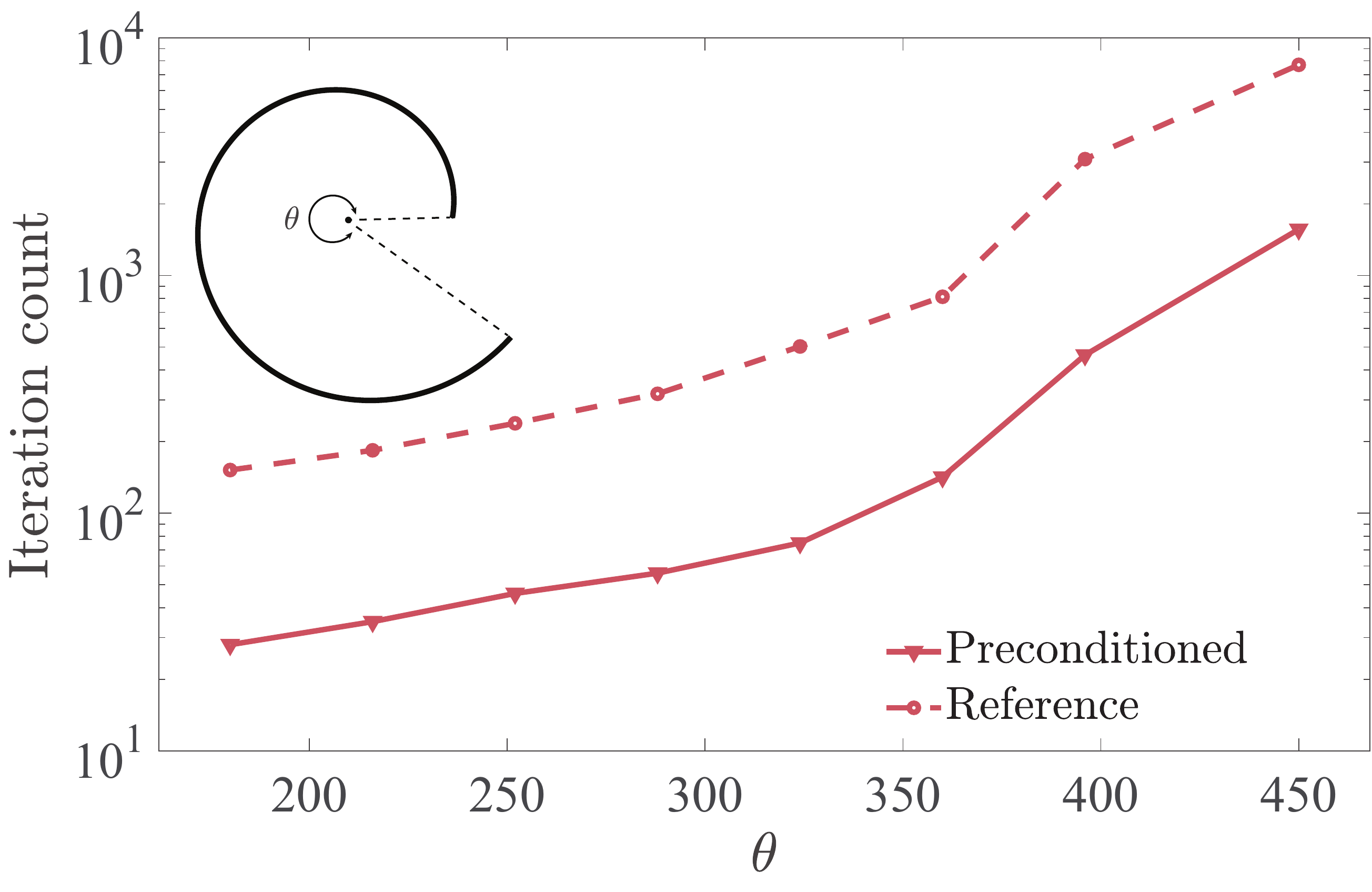} &  \includegraphics[height=2in]{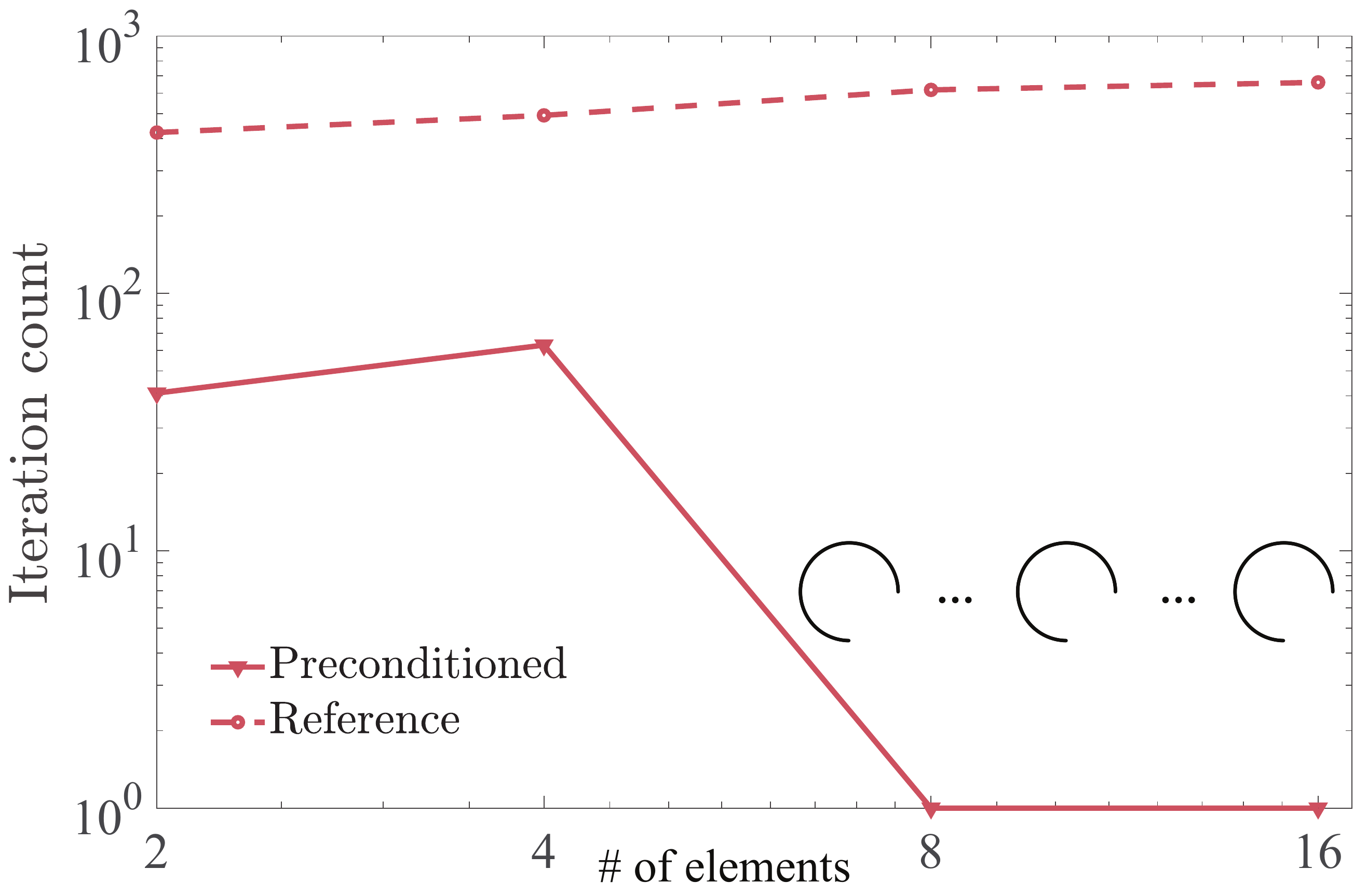} \\
            {(c)} & {(d)}         
        \end{tabular}
    \end{center}
    \caption{{Iteration counts with and without the proposed preconditioners for (a) curves with different shapes, (b) an open arc with varying angles, (c) an Archimedean spiral with varying rotation angles, (d) a 1-D array of open arcs with different element counts.}}
    \label{fig:iter}
\end{figure}

\begin{table}[ht!]
    \begin{center}
        \begin{tabular}{|c|c|c|c|c|c|c|}
            \hline
            shape & semicircle & corner & spiral & strips & square & cup \\
            \hline
            $N$=5E3 & 2.24E-06 & 9.51E-06 & 8.13E-06 & 7.12E-05 & 2.28E-05 & 1.60E-05 \\
            \hline            
            $N$=5E4 & 1.11E-05 & 9.84E-06 & 3.82E-05 & 6.45E-04 & 2.24E-04 & 1.84E-04 \\
            \hline
            $N$=5E5 & 5.86E-06 & 3.85E-06 & 4.01E-05 & 9.86E-04 & 2.38E-04 & 3.56E-04 \\
            \hline            
            $N$=5E6 & 1.10E-05 & 8.17E-06 & 1.37E-04 & 3.74E-04 & 2.33E-04 & 6.16E-04 \\
            \hline                            
        \end{tabular}
    \end{center}
    \caption{Measured solution error for different geometry shapes.}
    \label{tab:error}
\end{table}

\section{Conclusion and discussion}
\label{sec:conclusion}
This paper has introduced a simple and efficient algorithm, the hierarchical interpolative decomposition butterfly-LU factorization (H-IDBF-LU) preconditioner for solving two-dimensional electric-field integral equations (EFIEs) in electromagnetic scattering problems of perfect electrically conducting objects with open surfaces. H-IDBF-LU consists of two main parts: the first part applies the newly developed interpolative decomposition butterfly factorization (IDBF) to compress dense blocks of the discretized EFIE operator to expedite its application; the second part treats the lower and upper triangular part of the IDBF as an approximate LU factorization of the EFIE operator leading to an efficient preconditioner in iterative solvers.

Both the memory requirement and computational cost of the H-IDBF-LU solver scale as $O(N\log^2 N)$ in one iteration; the total number of iterations required for a reasonably good accuracy scales as $O(1)$ to $O(\log^2N)$ in all of our numerical tests. Our algorithm is simple to implement, automatically adapts to different structures with open surfaces, and is competitive with state-of-the-art MLFMA-based algorithms. A user-friendly MATLAB package, ButterflyLab (\url{https://github.com/ButterflyLab/ButterflyLab}), and a distributed parallel Fortran/C++ package, ButterflyPack (\url{https://github.com/liuyangzhuan/ButterflyPACK}), are freely available online.

The lower and upper triangular part of the EFIE operator can serve as an approximate LU factorization of the EFIE operator is an interesting observation and deserves much theoretical attention in the future. This idea could also be applied to three-dimensional EFIE's with open surfaces and we will explore this in future work.

{\bf Acknowledgments.} H. Yang was partially supported by Grant R-146-000-251-133 in the Department of
Mathematics at the National University of Singapore, by the Ministry of Education in Singapore under the
grant MOE2018-T2-2-147, and the start-up grant of the Department of Mathematics at Purdue University. Y. Liu thanks the support from the U.S. Department of Energy, Office of Science, Office of Advanced Scientific Computing Research, Scientific Discovery through Advanced Computing (SciDAC) program through the
FASTMath Institute under Contract No. DE-AC02-05CH11231 at Lawrence Berkeley National Laboratory. The authors also thank Dr. Xiaoye Sherry Li for useful discussions and suggestions on the paper.

\bibliographystyle{unsrt} 
\bibliography{ref}

\newpage
\appendix{IDBF}
\label{sec:IDBF}

\subsection{Overview}
\label{sub:ov}

Since the IDBF will be applied repeatedly in this paper, we briefly review the  $O(N\log N)$ IDBF algorithm proposed in \cite{IDBF} for a complementary low-rank matrix $K\in\mathbb{C}^{M\times N}$ with $M\approx N$ for the purpose of completeness. Let $X$ and $\Omega$ be the row and column index sets of $K$. Two trees $T_X$ and $T_\Omega$ of the same depth $L=O(\log N)$,
associated with $X$ and $\Omega$ respectively,
are constructed by dyadic partitioning with approximately equal node sizes with leaf node sizes no larger than $n_0$. Denote the root level of the tree as level $0$ and
the leaf one as level $L$. Such a matrix $K$ of size $M\times N$
is said to satisfy the {\bf complementary low-rank property} if for
any level $\ell$, any node $A$ in $T_X$ at level $\ell$, and any node
$B$ in $T_\Omega$ at level $L-\ell$, the submatrix $K_{A,B}$, obtained
by restricting $K$ to the rows indexed by the points in $A$ and the
columns indexed by the points in $B$, is numerically low-rank.
See Figure \ref{fig:submatrices} for an illustration of low-rank submatrices
in a complementary low-rank matrix of size $16n_0\times 16n_0$.

\begin{figure}[ht!]
  \begin{center}
    \begin{tabular}{ccccc}
      \includegraphics[height=1.1in]{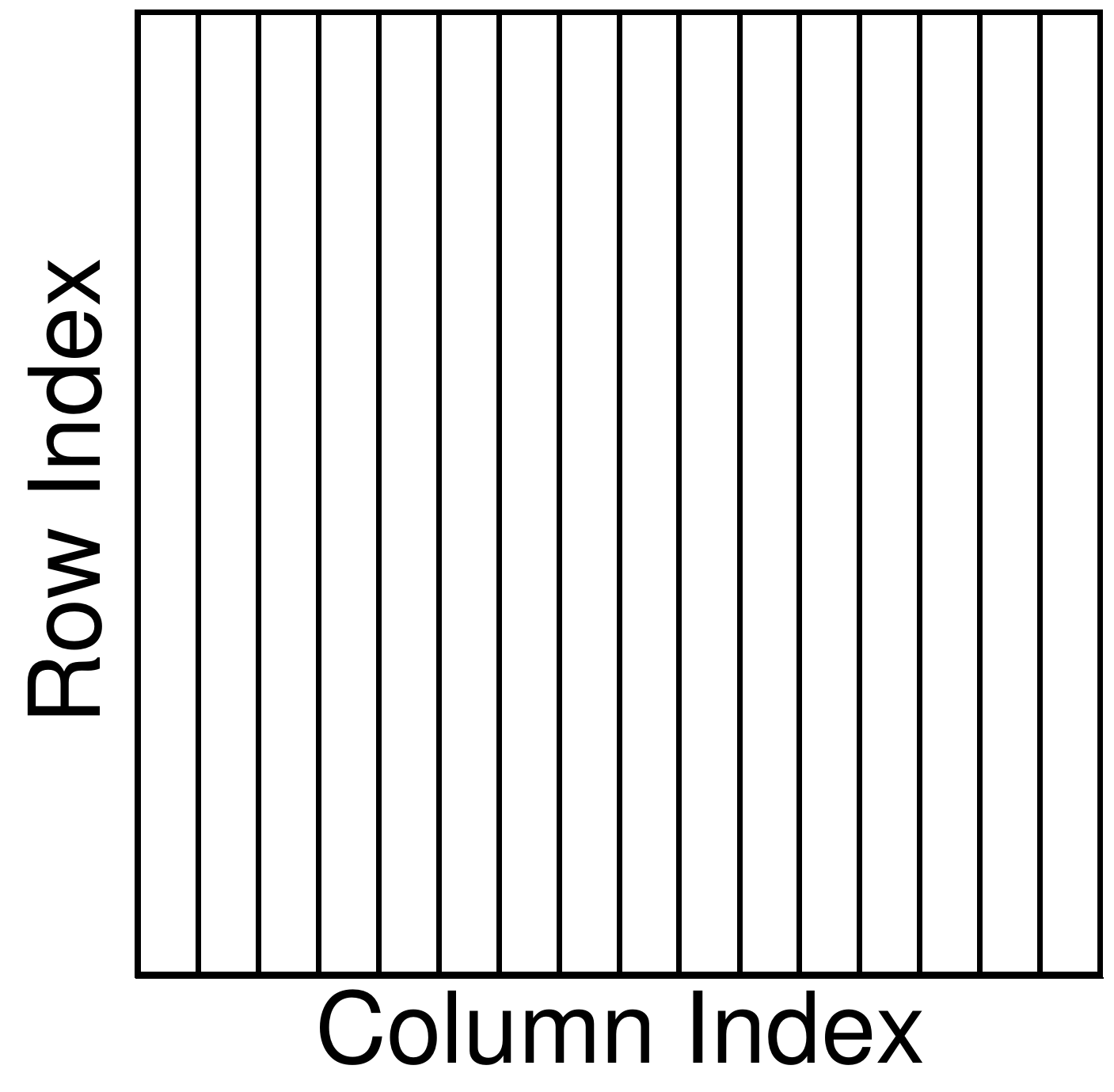}&
      \includegraphics[height=1.1in]{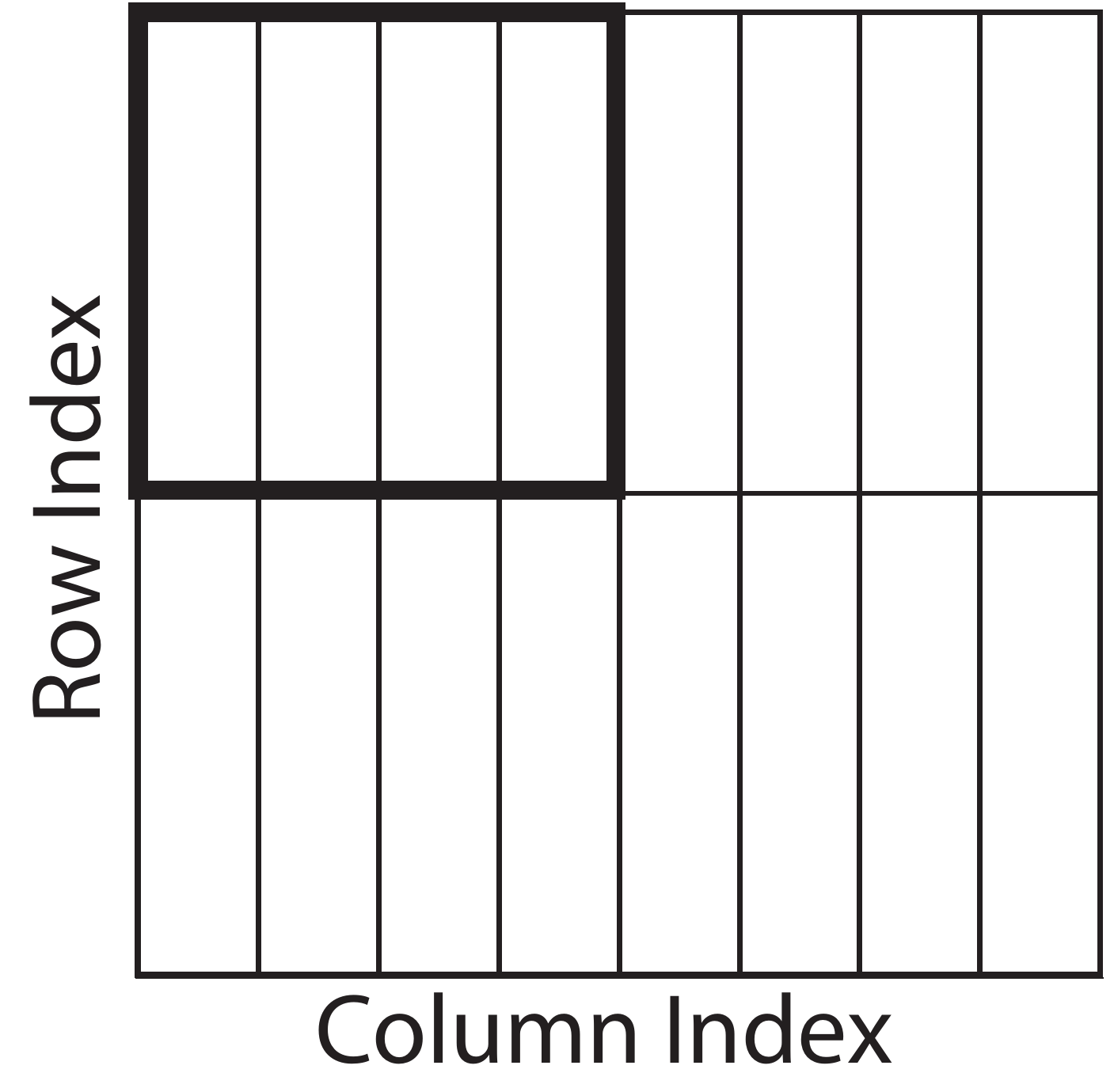}&
      \includegraphics[height=1.1in]{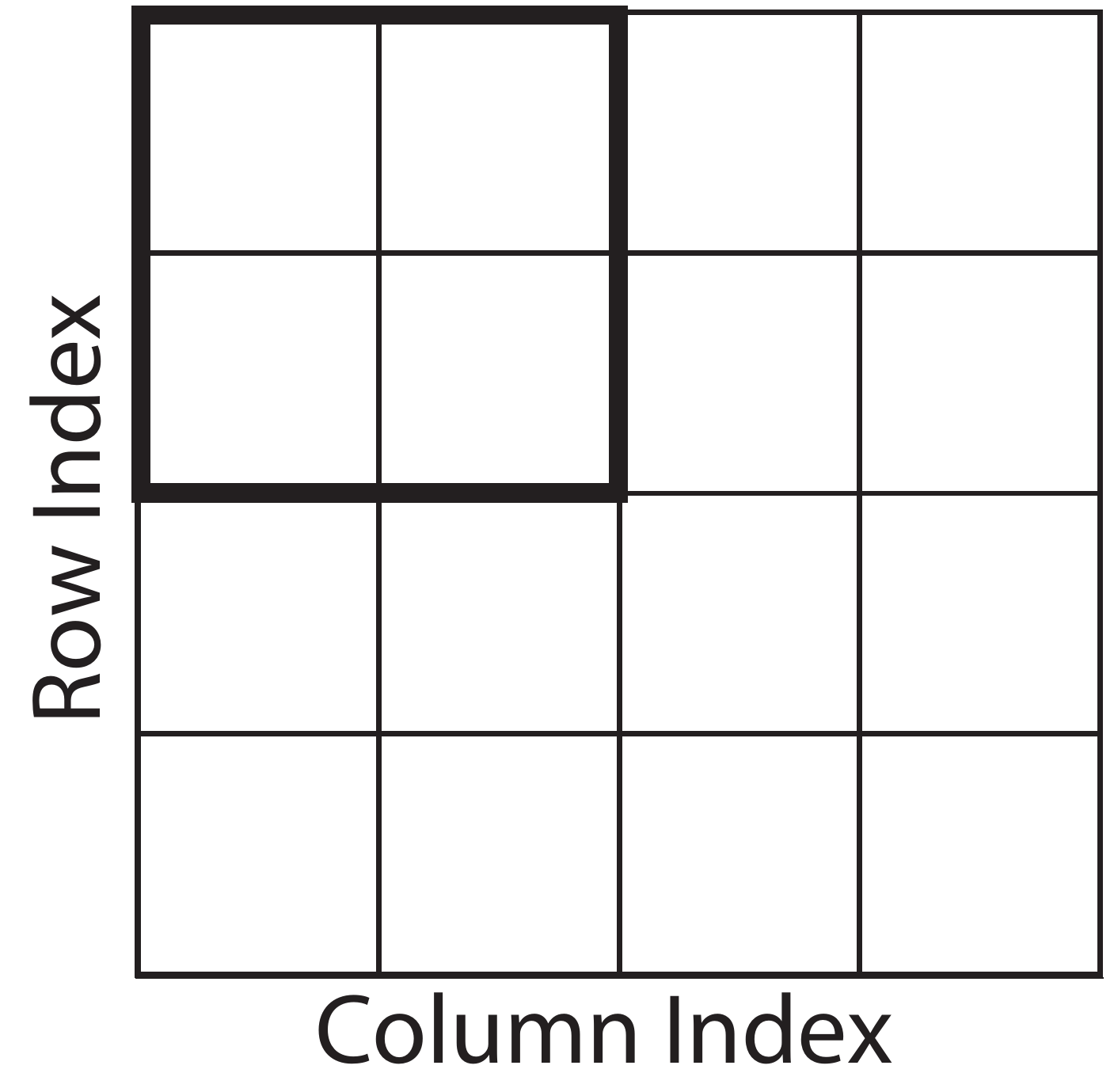}&
      \includegraphics[height=1.1in]{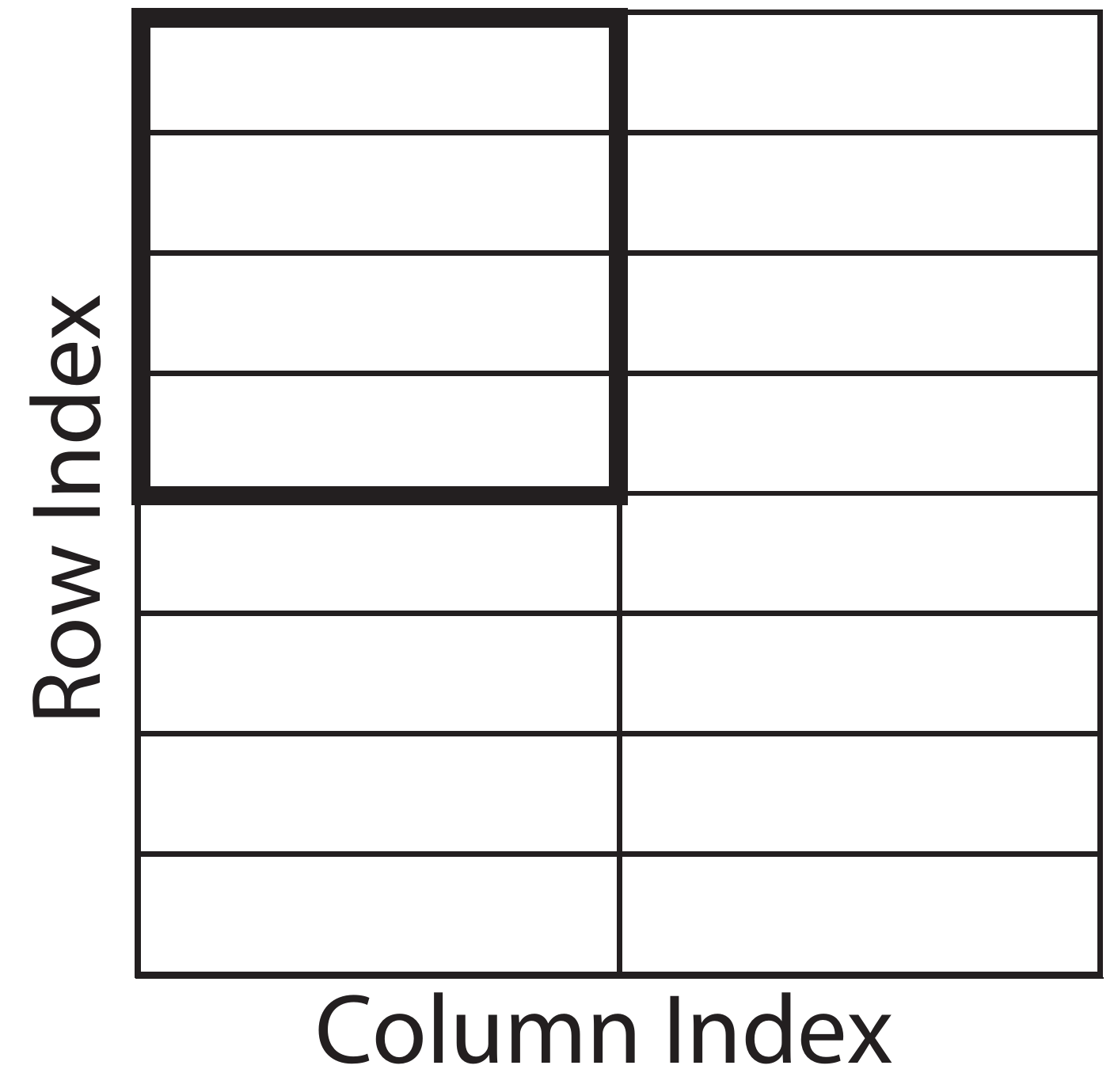}&
      \includegraphics[height=1.1in]{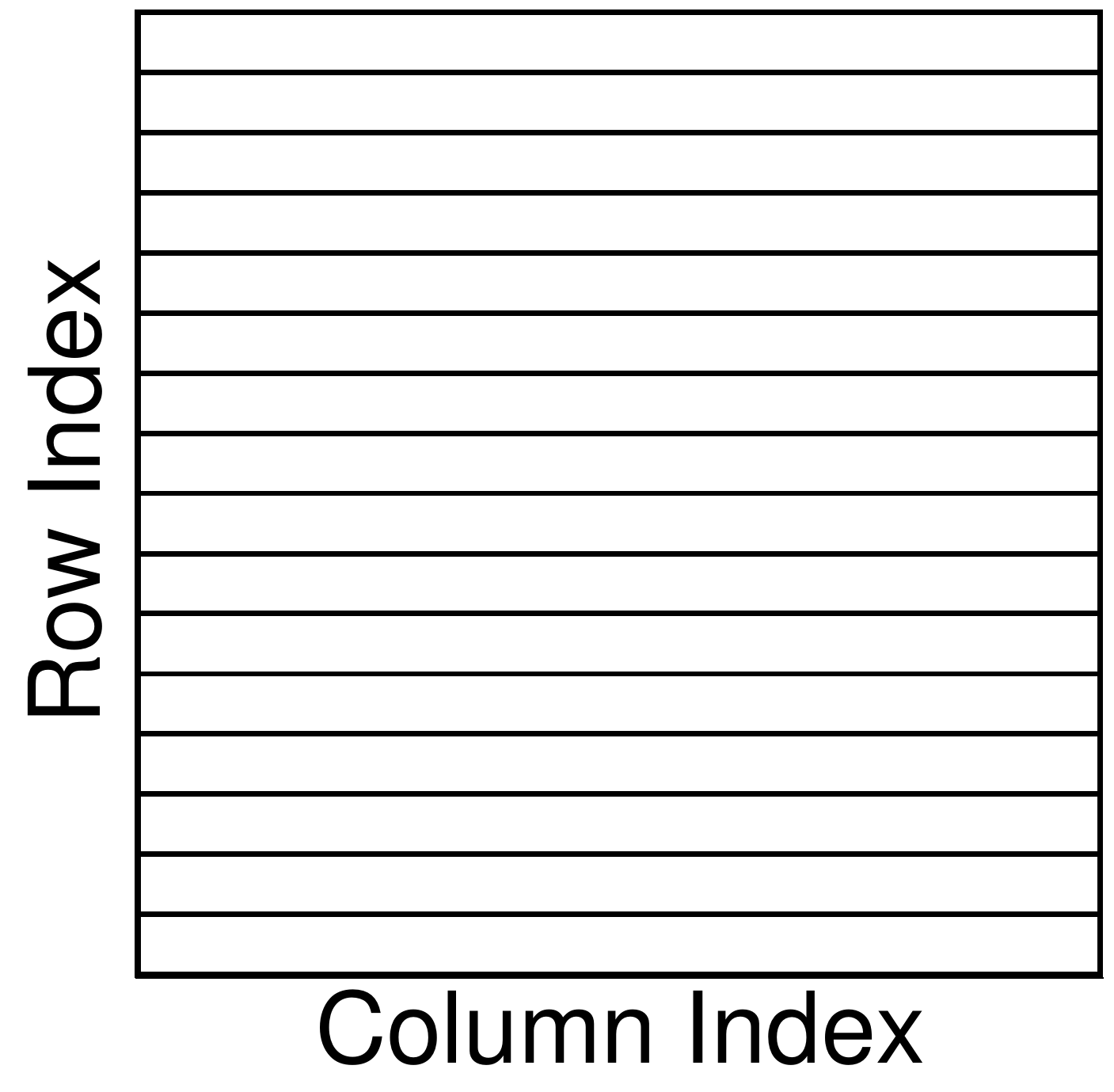}
    \end{tabular}
  \end{center}
  \caption{Hierarchical decomposition of the row and column indices
    of a complementary low-rank matrix of size
     $16n_0\times 16n_0$.  The trees $T_{X}$ ($T_{\Omega}$)
    has a root containing $16n_0$ column (row) indices and leaves
    containing $n_0$ row (column) indices.  The rectangles above
    indicate the low-rank submatrices that will be factorized in IDBF.}
\label{fig:submatrices}
\end{figure}

Given $K$, or equivalently an $O(1)$ algorithm to evaluate an arbitrary entry of $K$, IDBF aims at constructing a data-sparse representation of $K$ using the ID of low-rank submatrices in the complementary low-rank structure (see Figure \ref{fig:submatrices}) in the following form: 
\begin{equation}
\label{eqn:BFM}
K \approx U^{L}U^{L-1}\cdots U^{h} S^h V^{h}\cdots V^{L-1}V^{L},
\end{equation}
where the depth $L=\O(\log N)$ is assumed to be even without loss of generality,
$h=L/2$ is a middle level index, and all factors are sparse matrices with
$\O(N)$ nonzero entries. Storing and applying IDBF requires only $O(N\log N)$ memory and time. 

In what follows, uppercase letters will generally denote matrices, while the lowercase letters $c$, $p$, $q$, $r$, and $s$ denote ordered sets of indices. For a given index set $c$, its cardinality is written as $|c|$. Given a matrix $A$, $A_{pq}$, $A_{p,q}$, or $A(p,q)$ is the submatrix with rows and columns restricted to the index sets $p$ and $q$, respectively. We also use the notation $A_{:,q}$ to denote the submatrix with columns restricted to $q$. $s:t$ is an index set containing indices $\{s,s+1,s+2,\dots,t-1,t\}$. For the sake of simplicity, we assume that $N=2^Ln_0$, where $n_0=O(1)$ is the number of column or row indices in a leaf in the dyadic trees of row and column spaces, i.e., $T_{X}$ and $T_{\Omega}$, respectively. In practical numerical implementations, this assumption is not required.  

\subsection{Linear scaling Interpolative Decompositions}
\label{sub:ID}

Suppose $A$ has rank $k$, a rank revealing QR decomposition to $A_{s,:}$ gives
\begin{equation}
\label{eq:pivotedQR2}
A _{s,:}\Lambda = QR = Q[R_{1} \ R_{2}],
\end{equation}
 where $s$ is an index set {containing} $tk$ carefully selected rows of $A$ with $t$ as an oversampling parameter, $Q\in \bbC^{t{k\times {k}}}$ is an orthogonal matrix, $R\in \bbC^{{k}\times n}$ is upper trapezoidal, and $\Lambda \in \bbC^{n\times n}$ is a carefully chosen permutation matrix such that $R_{1}\in \bbC^{{k\times k}}$ is nonsingular. These $t{k}$ rows can be chosen from the Mock-Chebyshev
grids of the row indices as in \cite{MVBF,Mock1,Mock2}. 
 Let \begin{equation}
A_{s,q} = Q R_{1},\quad T = R_{1}^{-1}R_{2},
\end{equation}
then 
\begin{equation}
A_{:,p} \approx A_{:,q}T,
\end{equation}
which is equivalent to the traditional form of a column ID,
\begin{equation}
A \approx A_{:,q}[I \ T]\Lambda^*:=A_{:,q}V,
\end{equation}
where $q$ is the complementary set of $p$, $^*$ denotes the conjugate transpose of a matrix, and $V$ is the {\it column interpolation matrix}.. It can be easily shown that all the steps above require only $O({k}^2n)$ operations and $O({k}n)$ memory.

In practice, the true rank of $A$ is not available i.e., ${k}$ is unknown. As is in standard randomized algorithms, we could choose to fix a test rank ${k}\leq n$ or fix the approximation accuracy $\epsilon$ and find a numerical rank ${{k}_\epsilon}$ such that
\begin{equation}
\label{eq:req}
\|A-A_{:,q}V\|_2\leq O(\epsilon)
\end{equation}
with $T\in \mathbb{C}^{{{k}_\epsilon\times (n-{k}_\epsilon)}}$ and $V\in \mathbb{C}^{{{k}_\epsilon}\times n}$. We refer {to} this linear scaling column ID with an accuracy tolerance $\epsilon$ and a rank parameter ${k}$ as {\it $(\epsilon,{k})$-cID} ({\it $(\epsilon,{k})$-cID} for short).  For convenience, we will drop the term $(\epsilon,{k})$ when it is not necessary to specify it.

 Similarly, a row ID for the matrix $A\in \bbC^{m \times n}$  
\begin{equation}
A \approx \Lambda [I \ T]^* A_{q,:}:=UA_{q,:}
\end{equation}
can be attained by performing {\it cID} on $A^*$ with $O({k}^2m)$ operations and $O({k}m)$ memory. We refer {to} this linear scaling row ID as {\it ${(\epsilon,{k})}$-rID} and $U$ as the {\it row interpolation matrix}.

\subsection{Leaf-root complementary skeletonization (LRCS)}
\label{sub:LRCS}

Assume that at the leaf level of the row (and column) dyadic trees, the row index set $r$ (and the column index set $c$) of $A$ are divided into leaves $\{r_i\}_{1\leq i\leq m}$ (and $\{c_i\}_{1\leq i\leq m}$) in the following way:
\begin{equation}
\label{eq:partrc}
r = [r_{1},r_{2},\cdots,r_{m}] \qquad (\text{and } c = [c_{1},c_{2},\cdots,c_{m}]),
\end{equation}
with $|r_i|=n_0$ (and $|c_i|=n_0$) for all $1\leq i\leq m$, 
where $m = 2^{L}=\frac{N}{n_0}$, $L = \log_2 N - \log_2 n_0$, and $L+1$ is the depth of the dyadic trees $T_{X}$ (and $T_{\Omega}$). $\rID$ is applied to each $A_{r_{i},:}$ to compute the row interpolation matrix and denote it as $U_{i}$;  the associated skeleton indices is denoted as $\hat{r}_{i}\subset r_{i}$. Let $\hat{r} = [\hat{r}_{1},\hat{r}_{2},\cdots,\hat{r}_{m}]$, 
then $A_{\hat{r},:}$ is the important skeleton of $A$ and we can arrange all the small ID factors into a larger matrix factorization as follows:
\[
A\approx \begin{pmatrix}
U_{1} & & & \\
& U_{2} & & \\
& & \ddots & \\
& & & U_{m} 
\end{pmatrix}
\begin{pmatrix}
A_{\hat{r}_{1},c_{1}} & A_{\hat{r}_{1},c_{2}} & \hdots & A_{\hat{r}_{1},c_{m}} \\
A_{\hat{r}_{2},c_{1}} & A_{\hat{r}_{2},c_{2}} & \hdots & A_{\hat{r}_{2},c_{m}} \\
\vdots & \vdots & \ddots & \vdots \\
A_{\hat{r}_{m},c_{1}} & A_{\hat{r}_{m},c_{2}} & \hdots & A_{\hat{r}_{m},c_{m}}  
\end{pmatrix}
:=UM.
\]

Similarly, $\cID$ is applied to each $A_{\hat{r},c_{j}}$ to obtain the column interpolation matrix $V_{j}$ and the skeleton indices $\hat{c}_{j}\subset c_{j}$. Then finally we form the LRCS of $A$ as

\begin{equation}
\label{facA}
A\approx \begin{pmatrix}
U_{1} & & & \\
& U_{2} & & \\
& & \ddots & \\
& & & U_{m} 
\end{pmatrix}
\begin{pmatrix}
A_{\hat{r}_{1},\hat{c}_{1}} & A_{\hat{r}_{1},\hat{c}_{2}} & \hdots & A_{\hat{r}_{1},\hat{c}_{m}} \\
A_{\hat{r}_{2},\hat{c}_{1}} & A_{\hat{r}_{2},\hat{c}_{2}} & \hdots & A_{\hat{r}_{2},\hat{c}_{m}} \\
\vdots & \vdots & \ddots & \vdots \\
A_{\hat{r}_{m},\hat{c}_{1}} & A_{\hat{r}_{m},\hat{c}_{2}} & \hdots & A_{\hat{r}_{m},\hat{c}_{m}}  
\end{pmatrix}
\begin{pmatrix}
V_{1} & & & \\
& V_{2} & & \\
& & \ddots & \\
& & & V_{m} 
\end{pmatrix}:=USV.
\end{equation}
For a concrete example, Figure \ref{fig:rlfac} visualizes the non-zero pattern of the LRCS  in \eqref{facA}.

It is noteworthy that we only generate and store the skeleton of row and column index sets corresponding to $M$ and $S$, instead of computing $M$ and $S$ explicitly. Hence, it only takes $O(\frac{{k}^3}{n_0} N)$ operations and $O(\frac{{k}^2}{n_0}N)$ memory to generate and store the factorization in \eqref{facA}, since there are $2m=\frac{2N}{n_0}$ IDs in total.

\begin{figure}[htp]
\begin{minipage}{\textwidth}
\centering
\resizebox{3.5cm}{!}{
\begin{tikzpicture}[baseline=-0.5ex]
      \tikzset{every left delimiter/.style={xshift=-1ex},every right delimiter/.style={xshift=1ex}}
      \matrix (mat) [matrix of math nodes, left delimiter=(, right delimiter=)] {
\draw[fill=gray] (32,0) rectangle (0,32);
\\
      };
\end{tikzpicture}
}
$\approx$
\resizebox{1.85cm}{!}{
\begin{tikzpicture}[baseline=-0.5ex]
      \tikzset{every left delimiter/.style={xshift=-1ex},every right delimiter/.style={xshift=1ex}}
      \matrix (mat) [matrix of math nodes, left delimiter=(, right delimiter=)] {
      \draw;
\foreach \i in {0,1,...,15}{
   \draw[fill=gray]  (\i,32-\i-\i) rectangle (\i+1,32-\i-\i-2);
}\\
      };
\end{tikzpicture}
}
\resizebox{1.85cm}{!}{
\begin{tikzpicture}[baseline=-0.5ex]
      \tikzset{every left delimiter/.style={xshift=-1ex},every right delimiter/.style={xshift=1ex}}
      \matrix (mat) [matrix of math nodes, left delimiter=(, right delimiter=)] {
      \draw;
\foreach \i in {0,1,...,15}{
\foreach \j in {0,1,...,15}{
      \draw[fill=gray]  (\i,16-\j) rectangle (\i+1,16-\j-1);
}
}\\
      };
\end{tikzpicture}
}
\resizebox{3.5cm}{!}{
\begin{tikzpicture}[baseline=-0.5ex]
      \tikzset{every left delimiter/.style={xshift=-1ex},every right delimiter/.style={xshift=1ex}}
      \matrix (mat) [matrix of math nodes, left delimiter=(, right delimiter=)] {
      \draw;
\foreach \i in {0,1,...,15}{
   \draw[fill=gray]  (32-\i-\i,\i) rectangle (32-\i-\i-2,\i+1);
}\\
      };
\end{tikzpicture}
}\\
\end{minipage}
\caption{An example of the LRCS in \eqref{facA} of the complementary low-rank matrix $A$. Non-zero submatrices in \eqref{facA} are shown in gray areas.}
\label{fig:rlfac}
\end{figure}
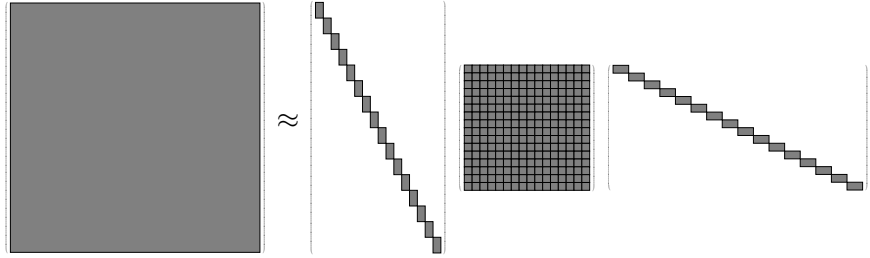

\subsection{Matrix splitting with complementary skeletonization (MSCS)}
\label{sub:MSCS}

 A complementary low-rank matrix $A$ (with row and column dyadic trees $T_{X}$ and $T_{\Omega}$ of depth $L+1$ and with $m=2^L$ leaves) can be split into a $2\times 2$ block matrix, $
A = 
\begin{pmatrix}
A_{11} & A_{12} \\
A_{21} & A_{22}
\end{pmatrix}$,  
according to the nodes of the second level of the dyadic trees $T_X$ and $T_\Omega$ (right next to the root). As a result, each $A_{ij}$ is also a complementary low-rank matrix with its row and column dyadic trees of $L-1$ levels. For example, $A$ in Figure \ref{fig:submatrices} is a five-level complementary low-rank matrix and $A_{11}$ is a three-level complementary low-rank matrix (see the highlighted submatrices in Figure \ref{fig:submatrices}).

Suppose {$A_{ij} \approx U_{ij}S_{ij}V_{ij}$}, for $i,j = 1,2$, is the LRCS of $A_{ij}$.
Then $A \approx USV$, where
\begin{equation}
\label{eq:expressUSV}
\begin{split}
U = 
\begin{pmatrix}
U_{11} & & U_{12} & \\
& U_{21} & & U_{22}
\end{pmatrix},\quad
S = \begin{pmatrix}
S_{11} & & & \\
& & S_{21} & \\
& S_{12} & & \\
& & & S_{22} 
\end{pmatrix},\quad
V = 
\begin{pmatrix}
V_{11} & \\
& V_{12} \\
V_{21} & \\
& V_{22}
\end{pmatrix}.
\end{split}
\end{equation}
The factorization in \eqref{eq:expressUSV} is referred as the MSCS. Recall that the middle factor $S$ is not explicitly computed, resulting in a linear scaling algorithm for forming \eqref{eq:expressUSV}. Figure \ref{fig:exlfac} visualizes the MSCS of a complementary low-rank matrix $A$.

\begin{figure}[htp]
\begin{minipage}{\textwidth}
\centering
\resizebox{3.5cm}{!}{
\begin{tikzpicture}[baseline=-0.5ex]
      \tikzset{every left delimiter/.style={xshift=-1ex},every right delimiter/.style={xshift=1ex}}
      \matrix (mat) [matrix of math nodes, left delimiter=(, right delimiter=)] {
\draw[fill=gray] (32,0) rectangle (0,32);
\\
      };
\end{tikzpicture}
}
$\approx$
\resizebox{3.5cm}{!}{
\begin{tikzpicture}[baseline=-0.5ex]
      \tikzset{every left delimiter/.style={xshift=-1ex},every right delimiter/.style={xshift=1ex}}
      \matrix (mat) [matrix of math nodes, left delimiter=(, right delimiter=)] {
      \draw;
\foreach \j in {0,1,...,15}{
    \draw[fill=gray] (\j,32-\j-\j) rectangle (\j+1,32-\j-\j-2);
    \draw[fill=gray] (16+\j,32-\j-\j) rectangle (16+\j+1,32-\j-\j-2);
}
\\
      };
\end{tikzpicture}
}
\resizebox{3.5cm}{!}{
\begin{tikzpicture}[baseline=-0.5ex]
      \tikzset{every left delimiter/.style={xshift=-1ex},every right delimiter/.style={xshift=1ex}}
      \matrix (mat) [matrix of math nodes, left delimiter=(, right delimiter=)] {
      \draw;
      \draw[fill=gray] (0,32) rectangle (8,24);
      \draw[fill=gray] (16,16) rectangle (24,24); 
      \draw[fill=gray] (8,8) rectangle (16,16);  
      \draw[fill=gray] (24,0) rectangle (32,8); 
\\
      };
\end{tikzpicture}
}
\resizebox{3.5cm}{!}{
\begin{tikzpicture}[baseline=-0.5ex]
      \tikzset{every left delimiter/.style={xshift=-1ex},every right delimiter/.style={xshift=1ex}}
      \matrix (mat) [matrix of math nodes, left delimiter=(, right delimiter=)] {
      \draw;
\foreach \j in {0,1,...,15}{
    \draw[fill=gray] (\j+\j,32-\j) rectangle (\j+\j+2,32-\j-1);
    \draw[fill=gray] (\j+\j,32-\j-16) rectangle (\j+\j+2,32-\j-17);
}
\\
      };
\end{tikzpicture}
}\\
\end{minipage}
\caption{The visualization of a MSCS of a complementary low-rank matrix $A \approx USV$. Non-zero blocks in \eqref{eq:expressUSV} are shown in gray areas.}
\label{fig:exlfac}
\end{figure}
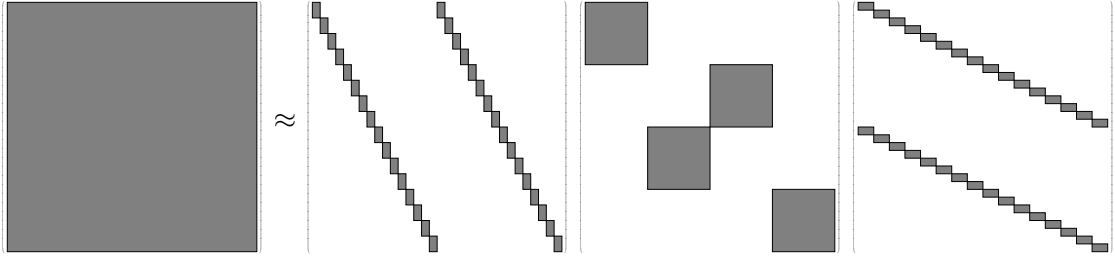

\subsection{Recursive MSCS}
\label{sub:RMSCS}

Now we apply MSCS recursively to get the full
IDBF. Suppose we have the first level of MSCS as
$A \approx U^{L}S^{L}V^{L}$ with
\begin{equation}
\label{eqn:USV2}
\begin{split}
U^{L} = 
\begin{pmatrix}
U^{L}_{11} & & U^{L}_{12} & \\
& U^{L}_{21} & & U^{L}_{22}
\end{pmatrix},\quad
S^{L} = \begin{pmatrix}
S^{L}_{11} & & & \\
& & S^{L}_{21} & \\
& S^{L}_{12} & & \\
& & & S^{L}_{22} 
\end{pmatrix},\quad
V^{L} = 
\begin{pmatrix}
V^{L}_{11} & \\
& V^{L}_{12} \\
V^{L}_{21} & \\
& V^{L}_{22}
\end{pmatrix}.
\end{split}
\end{equation}

By construction, we know $S^{L}_{ij}$ are complementary low-rank. Next, apply MSCS to each $S^L_{ij}$:
\begin{equation}
\label{eq:facSL1234}
{S^{L}_{ij} \approx U^{L-1}_{ij}S^{L-1}_{ij}V^{L-1}_{ij},}
\end{equation}
where 
\begin{equation}
\label{eq:expressUSVL-1}
\begin{split}
U^{L-1}_{ij} &= 
\begin{pmatrix}
(U^{L-1}_{ij})_{11} & & (U^{L-1}_{ij})_{12} & \\
& (U^{L-1}_{ij})_{21} & & (U^{L-1}_{ij})_{22}
\end{pmatrix},\\
S^{L-1}_{ij} &= 
\begin{pmatrix}
(S^{L-1}_{ij})_{11} & & & \\
& & (S^{L-1}_{ij})_{21} & \\
& (S^{L-1}_{ij})_{12} & & \\
& & & (S^{L-1}_{ij})_{22}
\end{pmatrix},\\
V^{L-1}_{ij} &= 
\begin{pmatrix}
(V^{L-1}_{ij})_{11} & \\
& (V^{L-1}_{ij})_{12} \\
(V^{L-1}_{ij})_{21} & \\
& (V^{L-1}_{ij})_{22}
\end{pmatrix}.
\end{split}
\end{equation}
 
Finally, organizing \eqref{eq:facSL1234} forms $S^{L} \approx U^{L-1}S^{L-1}V^{L-1}$ (see its visualization in Figure \ref{fig:exsfac}), where 
\begin{equation}
\label{eq:expressUSV_L-1}
U^{L-1} = 
\begin{pmatrix}
U^{L-1}_{11} & & & \\
& U^{L-1}_{21} & & \\
& & U^{L-1}_{12} & \\
& & & U^{L-1}_{11}
\end{pmatrix},\quad
S^{L-1} = 
\begin{pmatrix}
S^{L-1}_{11} & & & \\
& & S^{L-1}_{21} & \\
& S^{L-1}_{12} & & \\
& & & S^{L-1}_{22}
\end{pmatrix},
\end{equation}

\begin{equation*}
V^{L-1} = 
\begin{pmatrix}
V^{L-1}_{11} & & & \\
& V^{L-1}_{12} & & \\
& & V^{L-1}_{21} & \\
& & & S^{L-1}_{22}
\end{pmatrix},
\end{equation*}
leading to a second level factorization of $A$ as $A\approx U^L U^{L-1}S^{L-1}V^{L-1}V^L$.

\begin{figure}[htp]
\begin{minipage}{\textwidth}
\centering
\resizebox{3.5cm}{!}{
\begin{tikzpicture}[baseline=-0.5ex]
      \tikzset{every left delimiter/.style={xshift=-1ex},every right delimiter/.style={xshift=1ex}}
      \matrix (mat) [matrix of math nodes, left delimiter=(, right delimiter=)] {
      \draw;
      \draw[fill=gray] (0,32) rectangle (8,24);
      \draw[fill=gray] (16,16) rectangle (24,24); 
      \draw[fill=gray] (8,8) rectangle (16,16);  
      \draw[fill=gray] (24,0) rectangle (32,8); 
\\
      };
\end{tikzpicture}
}
$\approx$
\resizebox{3.5cm}{!}{
\begin{tikzpicture}[baseline=-0.5ex]
      \tikzset{every left delimiter/.style={xshift=-1ex},every right delimiter/.style={xshift=1ex}}
      \matrix (mat) [matrix of math nodes, left delimiter=(, right delimiter=)] {
      \draw;
\foreach \i in {0,8,...,24}{
\foreach \j in {0,1,...,3}{
    \draw[fill=lightgray] (\i+\j,32-\i-\j-\j) rectangle (\i+\j+1,32-\i-\j-\j-2);
    \draw[fill=lightgray] (\i+\j+4,32-\i-\j-\j) rectangle (\i+\j+5,32-\i-\j-\j-2);
}
}\\
      };
\end{tikzpicture}
}
\resizebox{3.5cm}{!}{
\begin{tikzpicture}[baseline=-0.5ex]
      \tikzset{every left delimiter/.style={xshift=-1ex},every right delimiter/.style={xshift=1ex}}
      \matrix (mat) [matrix of math nodes, left delimiter=(, right delimiter=)] {
      \draw;
\draw[fill=lightgray] (0,32) rectangle (2,30);
\draw[fill=lightgray] (4,30) rectangle (6,28);
\draw[fill=lightgray] (2,28) rectangle (4,26);
\draw[fill=lightgray] (6,26) rectangle (8,24);
\draw[fill=lightgray] (16,24) rectangle (18,22);
\draw[fill=lightgray] (20,22) rectangle (22,20);
\draw[fill=lightgray] (18,20) rectangle (20,18);
\draw[fill=lightgray] (22,18) rectangle (24,16);
\draw[fill=lightgray] (8,16) rectangle (10,14);
\draw[fill=lightgray] (12,14) rectangle (14,12);
\draw[fill=lightgray] (10,12) rectangle (12,10);
\draw[fill=lightgray] (14,10) rectangle (16,8);
\draw[fill=lightgray] (24,8) rectangle (26,6);
\draw[fill=lightgray] (28,6) rectangle (30,4);
\draw[fill=lightgray] (26,4) rectangle (28,2);
\draw[fill=lightgray] (30,2) rectangle (32,0);
\\
      };
\end{tikzpicture}
}
\resizebox{3.5cm}{!}{
\begin{tikzpicture}[baseline=-0.5ex]
      \tikzset{every left delimiter/.style={xshift=-1ex},every right delimiter/.style={xshift=1ex}}
      \matrix (mat) [matrix of math nodes, left delimiter=(, right delimiter=)] {
      \draw;
\foreach \i in {0,8,...,24}{
\foreach \j in {0,1,...,3}{
    \draw[fill=lightgray] (\i+\j+\j,32-\i-\j) rectangle (\i+\j+\j+2,32-\i-\j-1);
    \draw[fill=lightgray] (\i+\j+\j,32-\i-\j-4) rectangle (\i+\j+\j+2,32-\i-\j-5);
}
}\\
      };
\end{tikzpicture}
}\\
\end{minipage}
\caption{The visualization of the recursive MSCS of $S^L=U^{L-1}S^{L-1}V^{L-1}$.}
\label{fig:exsfac}
\end{figure}
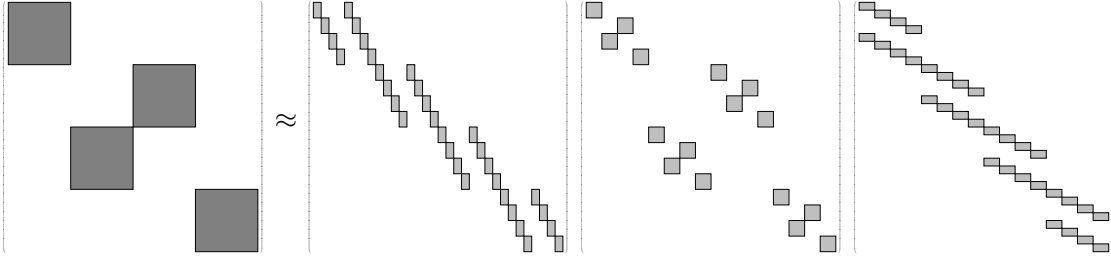

Similarly, we can apply MSCS recursively to each $S^{\ell}$ and assemble matrix factors hierarchically for $\ell=L$, $L-1$, $\dots$, $L/2$ to obtain
\begin{equation}
\label{eq:finisH-IDBF}
A \approx U^{L}U^{L-1}\cdots U^{h} S^h V^{h}\cdots V^{L-1}V^{L},
\end{equation}
where $h=L/2$. In the entire computation procedure, linear IDs only require $O(1)$ operations for each low-rank submatrix, and hence at most $O(N)$ for each level of factorization, and $O(N\log N)$ for the whole IDBF.

\end{document}